\documentclass[11pt,reqno]{amsart}
\usepackage{fullpage}
\usepackage{xcolor}
\usepackage[T1]{fontenc}                          
\usepackage{amsfonts}
\usepackage[utf8]{inputenc}                       
\usepackage{comment}                              
\usepackage{mparhack}                             
\usepackage{amsmath,amssymb,amsthm,mathrsfs,eucal}
\usepackage{enumerate}
\usepackage{enumitem}
\usepackage{booktabs} 
\usepackage{cancel}
\usepackage{graphicx,subfig}                    
\usepackage{wrapfig}                            
\usepackage[bookmarks=true, colorlinks=true]{hyperref}
\hypersetup{urlcolor=blue,citecolor=red,linkcolor=black}

\usepackage{bm}                                   
\usepackage{dsfont}

\newtheorem{defn}{Definition}[section]
\newtheorem{thm}{Theorem}[section]
\newtheorem{prop}{Proposition}[section]
\newtheorem{lem}{Lemma}[section]
\newtheorem{cor}{Corollary}[section]

\newtheorem{rem}{Remark}[section]

\DeclareMathOperator*{\argmin}{argmin}

\newcommand{\mP}{{\mathcal{P}}}
\newcommand{\mf}{{\mathcal{F}}}

\newcommand{\mh}{\mathcal{H}}
\newcommand{\mk}{{\mathcal{K}}}
\newcommand{\mee}{\mathrm{E}}

\newcommand{\mptra}{{\mathcal{P}_2^a(\R)}}
\newcommand{\mptrd}{{\mathcal{P}_2(\Rd)}}

\newcommand{\mpdtard}{{\mathcal{P}_{2}^a(\Rd)}}
\newcommand{\mw}{\mathcal{W}}

\newcommand{\rhote}{\rho_{\tau}^{\varepsilon}}

\newcommand{\rhotne}{\rho_{\tau,\varepsilon}^{n}}

\newcommand{\rhotnne}{\rho_{\tau,\varepsilon}^{n+1}}

\newcommand{\R}{\mathbb{R}}
\newcommand{\Rd}{{\mathbb{R}^{d}}}
\newcommand{\Rn}{{\mathbb{R}^{n}}}
\newcommand{\Rdd}{{\mathbb{R}^{2d}}}

\newcommand{\rhoe}{\rho^\varepsilon}

\newcommand{\rhoo}{\rho_{1}}
\newcommand{\rhotwo}{\rho_{2}}

\newcommand{\rrho}{\bm{\rho}}
\newcommand{\mmu}{\bm{\mu}}
\newcommand{\nnu}{\bm{\nu}}

\newcommand{\vv}{\bm{v}}

\newcommand{\rhoi}{\rho_{i}}

\newcommand{\rhootaue}{\rho_{1,\tau}^{\varepsilon}}
\newcommand{\rrhotau}{\bm{\rho}_{\tau}}
\newcommand{\rrhotaue}{\bm{\rho}_{\tau}^\varepsilon}

\newcommand{\rhottaue}{\rho_{2,\tau}^{\varepsilon}}

\newcommand{\rhootaune}{\rho_{1,\tau}^{\varepsilon, n}}
\newcommand{\rhottaune}{\rho_{2,\tau}^{\varepsilon, n}}

\newcommand{\rhootaunne}{\rho_{1,\tau}^{\varepsilon, n+1}}

\newcommand{\rhottaunne}{\rho_{2,\tau}^{\varepsilon, n+1}}

\newcommand{\rhoitaune}{\rho_{i,\tau}^{\varepsilon, n}}

\newcommand{\rhoitaunne}{\rho_{i,\tau}^{\varepsilon, n+1}}

\newcommand{\rhojtaune}{\rho_{j,\tau}^{\varepsilon, n}}

\newcommand{\rrhotaune}{\bm{\rho}_{\tau}^{\varepsilon,n}}
\newcommand{\rrhotaunne}{\bm{\rho}_{\tau}^{\varepsilon, n+1}}

\newcommand{\rhooe}{\rho_1^\varepsilon}
\newcommand{\rhotwoe}{\rho_2^\varepsilon}

\newcommand{\voe}{v_1^\varepsilon}
\newcommand{\vte}{v_2^\varepsilon}

\makeatletter
\@namedef{subjclassname@2020}{\textup{2020} Mathematics Subject Classification}
\makeatother

\def\XXint#1#2#3{{\setbox0=\hbox{$#1{#2#3}{\int}$}
  \vcenter{\hbox{$#2#3$}}\kern-.5\wd0}}

\title{Porous medium equation and cross-diffusion systems as limit of nonlocal interaction}
\date{}

\begin{document}
\author{Martin Burger \and Antonio Esposito}
\address{M. Burger -- 
Department Mathematik, Friedrich-Alexander-Universit\"at Erlangen-N\"urnberg, Cauerstrasse 11,
91058 Erlangen, Germany.}
\email{martin.burger@fau.de}

\address{A. Esposito -- Mathematical Institute, University of Oxford, Woodstock Road, Oxford, OX2 6GG, United Kingdom.}
\email{antonio.esposito@maths.ox.ac.uk}

\begin{abstract}
This paper studies the derivation of the quadratic porous medium equation and a class of cross-diffusion systems from nonlocal interactions. We prove convergence of solutions of a nonlocal interaction equation, resp. system, to solutions of the quadratic porous medium equation, resp. cross-diffusion system, in the limit of a localising interaction kernel. The analysis is carried out at the level of the (nonlocal) partial differential equations and we use gradient flow techniques to derive bounds on energy, second order moments, and logarithmic entropy. The dissipation of the latter yields sufficient regularity to obtain compactness results and pass to the limit in the localised convolutions. The strategy we propose relies on a discretisation scheme, which can be slightly modified in order to extend our result to PDEs without gradient flow structure. In particular, it does not require convexity of the associated energies. Our analysis allows to treat the case of limiting weak solutions of the non-viscous porous medium equation at relevant low regularity, assuming the initial value to have finite energy and entropy. 
\end{abstract}






\maketitle

\section{Introduction}
In this manuscript we deal with the connection between nonlocal interaction and the quadratic porous medium equation \eqref{eq:pme2}, as well as a class of cross-diffusion systems \eqref{eq:cross-diff-sys}, for a suitable choice of the interaction potentials. We show that a weak solution of the quadratic porous medium equation can be obtained as limit of a sequence of weak measure solutions of a nonlocal interaction equation; this can be extended to a class of cross-diffusion systems. More precisely, starting with the case of \eqref{eq:pme2} for ease of presentation, let $W_1:=V_1*V_1$, for a function $V_1$ satisfying some assumptions that will be clarified later, \textit{cf.}~\ref{ass:v1}. For any $\varepsilon>0$, consider the scaling $$W_\varepsilon(x)=\varepsilon^{-d}W_1(x/\varepsilon), \quad \mbox{i.e.} \quad W_\varepsilon=V_\varepsilon*V_\varepsilon.$$ We prove that, as $\varepsilon\to0^+$, a sequence of weak measure solutions to
\begin{equation}\label{eq:nlie}
    \partial_t\rhoe=\nabla\cdot(\rhoe\nabla W_\varepsilon*\rhoe) \tag{NLIE}
\end{equation}
converges to a weak solution of
\begin{equation}\label{eq:pme2}
    \partial_t\rho=\frac{1}{2}\Delta(\rho^2)=\nabla\cdot(\rho\nabla\rho). \tag{PME}
\end{equation}
An analogous result in the case of multi-species leads to a class of cross diffusion systems
\begin{align}
    \partial_{t}\rho_i=\sum_{j=1}^M \mbox{div}\left(\rho_i A_{ij}\nabla  \rho_j \right), \tag{CDS}
\end{align}
for $i,j=1,\dots,M$, $M\in\mathbb{N}$, under suitable assumptions on the matrix of the coefficients, as well as interaction kernels in the nonlocal version.

The main motivation for this work is to provide further insights into the derivation of diffusion-type equations from a system of interacting particles, rather than a direct derivation as in continuum mechanics. We remind the reader to \cite{Vaz07} for a complete overview on the analysis of the porous medium equation. Obtaining a particle approximation for the partial differential equation under study is a fascinating and useful result for the analysis of PDEs, as it provides a rigourous derivation and way to construct solutions of PDEs, leading to well-posedness, as well as powerful numerical methods. We mention here the seminal works \cite{onsager1944,morrey1955,dobrushin} and the review \cite{golse}. In case of transport equations (without diffusion), e.g. \eqref{eq:nlie}, deterministic approaches represent a reasonable choice since weak measure solutions may exist, in particular \textit{particle solutions} in the form of an empirical measure
\[
\rho_t^N=\frac{1}{N}\sum_{i=1}^N\delta_{X_i(t)},
\]
where, for any $i=1,\dots,N$, $X_i(t)$ solves a suitable ODE. For instance, for \eqref{eq:nlie} we would have the ODEs
\[
\dot{X}_i(t)=-\frac{1}{N}\sum_{j}\nabla W_\varepsilon(X_i(t)-X_j(t)).
\]
For further details we refer the reader to \cite{CDFFLS,CCH14}, and to \cite{DFF,DFESPSCH21} in case of systems of nonlocal PDEs. The problem is substantially different when diffusion is present, initial values in the form of an empirical measure disperse. More precisely, starting from a Dirac delta as initial datum, we will see an immediate smoothing effect which excludes measure solutions. For this reason, deterministic particle approximations are challenging, even though numerical methods have been proposed in this direction. We mention \cite{Russo90,GosseToscani06} for one dimensional linear and nonlinear diffusion, respectively, and \cite{Degond_Mustieles90,Patacchini_blob19} in any dimension. 

A successful attempt to overtake the aforementioned difficulty is given by stochastic particles undergoing a Brownian motion. We start mentioning an inspiring work for our paper, \cite{FigPhi2008}, where Figalli and Philipowski deal with the viscous porous medium equation with exponent $m>1$. They obtain (very weak) solutions as limit of a sequence of distributions of the solutions to nonlinear stochastic differential equations. This generalises previous results by Oelschl\"{a}ger, \cite{oelschlaeger2001,morale2005interacting}, and Philipowski, cf.~\cite{FigPhi2008}, the latter concerning only the case $m=2$. As byproduct of their analysis, the authors of \cite{FigPhi2008} prove propagation of chaos, thus providing a connection between microscopic and macroscopic description. In \cite{Philipowski2007} the quadratic porous medium equation is derived from a stochastic mean field interacting particle system with the addition of a vanishing Brownian motion. The concept of solution used is that of strong $L^1$, following \cite{Vazquez92_PM_introduction}, which is not the one used in this work, \textit{cf.~ Definition \ref{def:sol-pme2}} below. We point out that our strategy is different from the aforementioned papers since it does not require the addition of higher regularity induced by (vanishing) viscosity, but is based on an optimal transport theory approach, using the $2$-Wasserstein gradient flow structure of the two equations.

As counterpoint to stochastic methods, in \cite{Ol}, Oelschl\"{a}ger proves for the first time a particle approximation for classical and positive solutions of \eqref{eq:pme2} in $\Rd$, and for weak solutions in one dimension. Still in one space dimension, \cite{Daneri_Radici_Runa_JDE} presents a deterministic particle approximation for aggregation-diffusion equations, including the porous medium equation. In the recent article \cite{Patacchini_blob19}, the authors provide a deterministic particle method for linear and nonlinear diffusion equations, interpreted as $2$-Wasserstein gradient flows. Their approach is inspired by the blob method for aggregation equations in \cite{Craig_blob2016}. In particular, in \cite{Patacchini_blob19}, Carrillo, Craig, and Patacchini proceed by regularising the associated internal energy and prove $\Gamma$-convergence towards the unregularised energy, both for linear and nonlinear diffusion, that is $m\ge1$. With the addition of a confining drift or interaction potential, they also show stability of minimisers, ensured by the additional potentials. In case $m\ge2$, they provide stability of gradient flows under sufficient regularity conditions, using the approach of Sandier and Serfaty, \cite{Sandier_Serfaty2004,Serfaty_gammaconv_2011}. For the quadratic porous medium equation, i.e. $m=2$, the regularity conditions needed are satisfied for an initial datum with bounded second order moments and entropy --- as in our case. This generalises a previous result by Lions and
Mas-Gallic, \cite{Lions_MasGallic_2001}, on a numerical scheme for \eqref{eq:pme2} on a bounded domain with periodic boundary conditions. In the case $m>2$ or more general initial data, it is an open problem to check and apply the stability in \cite[Theorem 5.8]{Patacchini_blob19}. Recently in \cite{blob_weighted_craig}, Craig et al. use the blob method to obtain a deterministic particle approximation for weighted (quadratic) porous medium equations, relevant in sampling methods, control theory, and in models of two-layer neural networks.

In the case of cross-diffusion systems, a stochastic approach has been recently considered by Chen et al. in \cite{Holzinger_deriv_cross_diff}, extending the mean-field limit studied in \cite{ChenDausJuengel2019}, the latter differing from \eqref{eq:cross-diff-sys} by the addition of linear diffusion in each species. The key idea in \cite{Holzinger_deriv_cross_diff} is to consider interacting diffusion coefficients in the systems of stochastic differential equations and to perform the mean-field limit using an intermediate nonlocal cross-diffusion system. In \cite{Moussa_nl_lo_triangular2020}, Moussa shows the nonlocal-to-local limit for the triangular SKT model on a torous, with bounded coefficients. The latter work partially addresses a question raised by Fonbona and Méléard in \cite{Fontbona_Meleard_nl_SKT}, where they consider a nonlocal version of the SKT model, \cite{SKT}.  Let us stress that \eqref{eq:cross-diff-sys} is different from the SKT model. To the best of our knowledge, a deterministic derivation has not been proven yet. 

Our result is strictly related to \cite{Patacchini_blob19} and \cite{blob_weighted_craig} as it provides a rigorous procedure to derive the quadratic porous medium equation from the nonlocal interaction equation. We observe that the regularisation of the energy in \cite{Patacchini_blob19} and \cite{blob_weighted_craig}, for $m=2$ and $\bar{\rho}\equiv1$, corresponds to our choice for the interaction potential in \eqref{eq:nlie}, though we relax the regularity assumption on the kernel in the convolution, so that to include Morse type potentials, \textit{cf.~}\ref{ass:v1}. Moreover, we propose an alternative approach that can be used without $\lambda$-convexity and even if a gradient flow structure is not exhibited (see below and section \ref{sec:conclusion}). More precisely, we construct solutions of \eqref{eq:nlie} by means of the JKO scheme, \cite{JKO}, in order to obtain uniform estimates on the sequence of solutions $\{\rhoe\}_\varepsilon$, which are nevertheless only measures. This issue is solved by considering a smoothed version of $\rho^\varepsilon$, given by $v^\varepsilon: = V_\varepsilon*\rho^\varepsilon $. Indeed, starting from an initial probability density in $L^2(\Rd)$ with finite second order moment and logarithmic entropy, we are able to prove a uniform bound in $H^1$ for $v_\varepsilon$, entailing the right compactness to pass to the limit in the weak formulations of the equations and recover a weak solution of \eqref{eq:pme2}. Our analysis mirrors that convergence from deterministic particle system might not work due to infinite entropy. This problem is also observed in \cite[Remark 6.3]{Patacchini_blob19}, though numerical simulations in \cite[Section 6]{Patacchini_blob19} give confidence that deterministic approximation could be achieved. Indeed, $\lambda$-convexity of the associated energy allows to exploit stability estimates with the $2$-Wasserstein distance so that to obtain a particle approximation when the number of particles involved depends on $\varepsilon$, i.e. $N=N(\varepsilon)$, in a way the approximation of the initial datum converges zero in Wasserstein fast enough --- this is proven in \cite[Theorem 1.4]{blob_weighted_craig} assuming more regularity on the kernel, at least $V\in C^2$. A result for the number of particles independent of the localisation scaling is still open. Interacting stochastic particle systems still represent a solid method. 

Our approach neither exploits  $\lambda$-convexity of the energies involved, as in \cite{Patacchini_blob19} and \cite{blob_weighted_craig}, nor any equivalent gradient flow formulation of the equations such as \textit{evolution variational inequality} or \textit{curve of maximal slope}. This is indeed an advantage, since using the JKO scheme at the level of the nonlocal interaction equation allows to extend our strategy to the case of equations which are not gradient flows by means of a suitable splitting scheme, \cite{CL2}. The latter issue is also relevant for the extension of our result to cross-diffusion systems, since geodesic convexity is valid in few cases, if any. Indeed, $\lambda$-convex gradient flows may be too restrictive since the corresponding assumption on the diffusion matrix effectively leads to diagonal diffusion, i.e. $A_{ij}=0$ for $j \neq i$, \textit{cf.~}\cite{matthes_zinsl, beck_matthes_zizza22}. Furthermore, in order to include a large class of cross-diffusion systems, we note that the corresponding nonlocal interaction system does not necessarily exhibit a Wasserstein gradient flow structure, \textit{cf.~}\cite{DFF}. However, this does not exclude to apply a time-discretisation of the system to get uniform bounds and existence of a (sequence) of solutions, as proven in section \ref{sec:systems}.

We also observe that our strategy can be applied to linear Fokker-Planck equations, for suitable assumptions on the external potentials, since we anyway need to assume finite logarithmic entropy initially. Similarly, one can add linear diffusion for each species in \eqref{eq:cross-diff-sys}. As previously mentioned, the extension to the non-viscous and \textit{non-quadratic} porous medium equation is still an open problem, \textit{cf.~}\cite[Theorem 5.8]{Patacchini_blob19}. It is then natural to see whether our approach can be used for $m\neq2$, using a different nonlocal equation. For a better understanding of these problems we provide more details in  section \ref{sec:conclusion}.

\subsection{Structure of the paper}
First, in section \ref{sec:notation} we specify the notation and preliminary concepts used throughout the paper. In section \ref{sec:nlie} we focus on the nonlocal interaction equation \eqref{eq:nlie}. We provide existence of a sequence of weak measure solutions, \textit{cf.~Definition \ref{def:weak-meas-sol}}, by means of the JKO scheme, which is useful to derive uniform estimates on the associated energy and second order moments. Section \ref{sec:compactness} is devoted to obtain the suitable compactness for the sequence of weak measure solutions to \eqref{eq:nlie}. In order to pass to the limit in the weak formulation of \eqref{eq:nlie} to obtain the weak solution of \eqref{eq:pme2} we derive a uniform $H^1$ bound (in space) on a suitable smoothed sequence associated, by taking advantage of the time-discretisation of \eqref{eq:nlie}. In view of this analysis, we recover the weak solution of \eqref{eq:pme2} in the $\varepsilon\to0^+$ limit in Theorem \ref{thm:existence-pme2} in section \ref{sec:local-limit-pme2}. The results obtained in the previous sections are extended to the multi-species case in section \ref{sec:systems}. We conclude the paper with some remarks on possible extensions of our result in section \ref{sec:conclusion}. 

\section{Notation and preliminaries}\label{sec:notation}

The interaction potential we consider in the one-species case is the (rescaled)  convolution $W_1:=V_1*V_1$, being $V_1:\Rd\to\R$ such that the following conditions hold:
\begin{enumerate}[label=\textbf{(V)}]
    \item\label{ass:v1} $V_1\in C_b(\Rd;[0,+\infty))\cap C^1(\Rd\setminus\{0\})$, $\|V_1\|_{L^1}=1$, $V_1(x)=V_1(-x)$, $\int_\Rd|x|V_1(x)\,dx<+\infty$, $\nabla V_1\in L^1(\Rd)$, and $|\nabla V_1(x)|\le C(1+|x|)$.  
\end{enumerate}
Among possible examples of kernels $V_1$, we mention Gaussians or pointy potentials such as Morse, the latter not covered by previous results. The assumption \ref{ass:v1} implies that the interaction potential satisfies
\begin{enumerate}[label=\textbf{(W)}]
    \item\label{ass:w} $W_1\in C(\Rd;[0,\infty))$, $W_1(x)=W_1(-x)$ for all $x\in\Rd$, $W_1\in C^1(\Rd\setminus\{0\})$ such that $\nabla W_1=\nabla V_1*V_1$ and $\nabla W_1\in L^1(\Rd)$.
\end{enumerate}
More precisely, we consider the rescaled functions $W_\varepsilon(x)=\varepsilon^{-d}W_1(x/\varepsilon)$ and $V_\varepsilon(x)=\varepsilon^{-d}V_1(x/\varepsilon)$, hence $W_\varepsilon=V_\varepsilon*V_\varepsilon$.

\begin{rem}
We observe that regularity for $W_1$ is inferred by $V_1$, following a standard proof where the lack of compact support can be overtaken by using boundedness from above of $V_1$ and Egorov's theorem. Continuity of partial derivatives is then obtained from $L^1$ integrability of $\nabla V_1$ and Lebesgue's dominated convergence theorem.
\end{rem}

Throughout the manuscript we will denote by $\mP(\Rd)$ the set of probability measures on $\Rd$, for $d\in\mathbb{N}$, and by $\mP_p(\Rd):=\{\rho\in\mP(\Rd):m_p(\rho)<+\infty\}$, being
$m_p(\rho):=\int_\Rd|x|^p\,d\rho(x)$ the $p^{\mathrm{th}}$-order moment of $\rho$, for $1\le p<\infty$. We shall use $\mP_p^a(\Rd)$ for elements in $\mP_p(\Rd)$ which are absolutely continuous with respect to the Lebesgue measure. For $p=2$, the $2$-Wasserstein distance between $\mu_1,\mu_2\in \mptrd$ is
\begin{equation}\label{wass}
d_W^2(\mu_1,\mu_2):=\min_{\gamma\in\Gamma(\mu_1,\mu_2)}\left\{\int_{\Rdd}|x-y|^2\,d\gamma(x,y)\right\},
\end{equation}
where $\Gamma(\mu_1,\mu_2)$ is the class of all transport plans between $\mu_1$ and $\mu_2$, that is the class of measures $\gamma\in\mP(\Rdd)$ such that, denoting by $\pi_i$ the projection operator on the $i$-th component of the product space, the marginality condition
$$
(\pi_i)_{\#}\gamma=\mu_i \quad \mbox{for}\ i=1,2
$$
is satisfied. In the expression above marginals are the push-forward of $\gamma$ through $\pi_i$. For a measure $\rho\in\mP(\Rd)$ and a Borel map $T:\Rd\to\Rn$, $n\in\mathbb{N}$, the push-forward of $\rho$ through $T$ is defined by
$$
 \int_{\Rn}f(y)\,dT_{\#}\rho(y)=\int_{\Rd}f(T(x))\,d\rho(x) \qquad \mbox{for all $f$ Borel functions on}\ \Rn.
$$
Setting $\Gamma_0(\mu_1,\mu_2)$ as the class of optimal plans, i.e. minimizers of \eqref{wass}, the $2$-Wasserstein distance can be written as
$$
d_W^2(\mu_1,\mu_2)=\int_{\Rdd}|x-y|^2\,d\gamma(x,y), \qquad \gamma\in\Gamma_0(\mu_1,\mu_2).
$$
We refer the reader to \cite{AGS,V2,S} for further details on optimal transport theory and Wasserstein spaces. 
\begin{rem}\label{rem:mom-ineq}
From the Definition of the $2$-Wasserstein distance and the inequality $|y|^2\le2|x|^2+2|x-y|^2$ it follows that
$$
m_2(\rho_1)\le2m_2(\rho_0)+2d_W^2(\rho_0,\rho_1), \qquad \forall \rho_0,\rho_1\in\mptrd.
$$
\end{rem}
In Proposition \ref{prop:aulirs-meas} we use the $1$-Wasserstein distance, denoted by $d_1$ and defined by
\begin{align}\label{eq:1wass}
    d_1(\mu_1,\mu_2):=\min_{\gamma\in\Gamma(\mu_1,\mu_2)}\left\{\int_{\Rdd}|x-y|\,d\gamma(x,y)\right\}.
\end{align}
Below we specify the concept of solution to the quadratic porous medium we consider, as well as that of weak measure solutions of the nonlocal interaction equation.
\begin{defn}[Weak solution to \eqref{eq:pme2}]\label{def:sol-pme2}
A weak solution to the porous medium equation
\begin{equation*}
    \begin{cases}
    \partial_t\rho=\nabla\cdot(\rho\nabla\rho)\\
    \rho(0,\cdot)=\rho_0
    \end{cases}\tag{PME}
\end{equation*}
on the time interval $[0,T]$ with initial datum $\rho_0\in\mpdtard\cap L^2(\Rd)$ such that $\int_\Rd\rho_0(x)\log\rho_0(x)dx<\infty$ is a curve $\rho\in C([0,T];\mP_2(\Rd))$ satisfying the following properties:
\begin{enumerate}
    \item for almost every $t\in[0,T]$ the measure $\rho(t)$ has a density with respect to the Lebesgue measure, still denoted by $\rho(t)$, and $\rho\in L^2([0,T];H^1(\Rd))$;
    \item for any $\varphi\in C^1_c(\R^d)$ and all $t\in[0,T]$ it holds
\begin{align*}
        \int_\Rd\varphi(x)\rho(t,x)\,dx= \int_\Rd\varphi(x)\rho_0(x)\,dx-\int_0^t\int_\Rd\rho(s,x) \nabla \varphi(x)\cdot \nabla \rho(s,x)\,dx\,ds.
    \end{align*}
\end{enumerate}
\end{defn}

\begin{defn}[Weak measure solution to \eqref{eq:nlie}]\label{def:weak-meas-sol}
A narrowly continuous curve $\rho^\varepsilon:[0,T]\to\mptrd$, mapping $t\in[0,T]\mapsto\rho_t^\varepsilon\in\mptrd$, is a weak measure solution to \eqref{eq:nlie} if, for every $\varphi\in C^1_c(\Rd)$ and any $t\in[0,T]$, it holds
\begin{equation}\label{eq:weak-form}
    \int_\Rd\varphi(x)d\rho_t^\varepsilon(x)\!-\!\int_\Rd\varphi(x)d\rho_0(x)=-\frac{1}{2}\int_0^t\!\iint_\Rdd(\nabla\varphi(x)\!-\!\nabla\varphi(y))\cdot\nabla W_\varepsilon(x-y)d\rho_r^\varepsilon(y)d\rho_r^\varepsilon(x)dr.
\end{equation}
\end{defn}

\begin{rem}
Our choice for the Definition of weak measure solution to \eqref{eq:nlie} strongly depends on the aim of our paper,  that is showing convergence of solutions of \eqref{eq:nlie} to weak solutions of \eqref{eq:pme2}, according to Definition \ref{def:sol-pme2}. Note that \eqref{eq:nlie} is a continuity equation of the form 
\[
\begin{cases}
\partial_t\rho_t+\nabla\cdot(\rho_t w_t^\varepsilon)=0\\
w_t^\varepsilon=-\nabla W_\varepsilon*\rho_t,
\end{cases}
\]
with a Borel velocity field such that, for $\varepsilon>0$,
\begin{align*}
\int_0^T\int_\Rd|w_t^\varepsilon(x)|\,d\rho_t(x)\,dt&=\int_0^T\int_\Rd|\nabla W_\varepsilon*\rho_t(x)|\,d\rho_t(x)\,dt\\
&\le\frac{CT}{\varepsilon^d}\|V_1\|_{L^1(\Rd)}+\frac{2C}{\varepsilon^{d+1}}\|V_1\|_{L^1(\Rd)}\int_0^T\int_{\Rd}|x|\,d\rho_t(x)\,dx\,dt\\
&\quad+\frac{CT}{\varepsilon^d}\int_\Rd|x|V_1(x)\,dx<+\infty,
\end{align*}
where we used the growth condition on $|\nabla V_\varepsilon|$ and preservation of second order moments (cf.~Lemma \ref{prop:en-ineq-mom-bound}). In turn, \cite[Lemma 8.2.1]{AGS} provides the existence of a continuous representative for distributional solutions of continuity equations with velocity fields in $L^1([0,T];L^1(\rho_t))$. In particular, for test functions time-independent, we get formulation \eqref{eq:weak-form}, where we also used that $\nabla W_\varepsilon$ is odd. Note that this formulation overtakes the loss of regularity at $0$ for $\nabla W_\varepsilon$, as already noticed in \cite{CDFFLS}.
\end{rem}

For the reader's convenience we postpone notations and preliminaries on the multi-species case to section \ref{sec:systems}.

\section{Results on the nonlocal interaction equation}\label{sec:nlie}

The nonlocal interaction equation has been intensively studied, especially in the context of $2$-Wasserstein gradient flows. In \cite{AGS}, the authors deal with \eqref{eq:nlie} for convex potentials that do not produce a blow-up in finite time. In case of more singular convex potentials, a well-posedness theory for weak measure solutions is given by \cite{CDFFLS}. Furthermore, it is worth to mention \cite{bertozzi} and the references therein, where $L^p$ theory for the aggregation equation is provided.

In this paper we consider an interaction potential satisfying assumptions similar to \cite{CDFFLS}, though not convex, with the aim of applying the JKO scheme, \cite{JKO}, in order to obtain a priori estimates on the solutions of \eqref{eq:nlie} and their smoothed version $v^\varepsilon=V_\varepsilon*\rhoe$. In turn, we are able to show convergence towards \eqref{eq:pme2}. The interaction potential we choose is the (rescaled) convolution $W_\varepsilon=V_\varepsilon*V_\varepsilon$, for $W_\varepsilon(x)=\varepsilon^{-d}W_1(x/\varepsilon)$, satisfying \ref{ass:w} that we recall here for convenience:
\begin{enumerate}[label=\textbf{(W)}]
    \item $W_\varepsilon\in C(\Rd;[0,\infty))$, $W_\varepsilon(x)=W_\varepsilon(-x)$ for all $x\in\Rd$, $W_\varepsilon\in C^1(\Rd\setminus\{0\})$ such that $\nabla W_\varepsilon=\nabla V_\varepsilon*V_\varepsilon$ and $\nabla W_\varepsilon\in L^1(\Rd)$.
\end{enumerate}
Let us emphasise that in this section $\varepsilon>0$ is fixed and finite. We assume the initial datum $\rho_0\in\mptrd\cap L^2(\Rd)$, and the interaction energy functional $\mw_\varepsilon:\mptrd\to(-\infty,+\infty]$ is given by
\[
\mw_\varepsilon[\rho]=\frac{1}{2}\int_{\Rd}(W_\varepsilon*\rho)(x)d\rho(x).
\]
\begin{rem}\label{rem:uniform-bound-initial-energy}
Using that $V_\varepsilon$ is even, we observe that the nonlocal interaction energy is nothing but the $L^2$ norm of the \textit{smoothed solution} $v^\varepsilon$, since
\[
\int_\Rd [V_\varepsilon*(V_\varepsilon*\rho)](x)d\rho(x)=\int_\Rd|(V_\varepsilon*\rho)(x)|^2dx.
\]
In view of the equivalence above, for $\rho_0\in\mptrd\cap L^2(\Rd)$ we have a uniform bound from above for the nonlocal interaction energy at the initial datum. More precisely,
\begin{align*}
    \mw_\varepsilon[\rho_0]&=\frac{1}{2}\int_{\Rd}(W_\varepsilon*\rho_0)(x)\rho_0(x)\,dx\\
    &=\frac{1}{2}\int_\Rd |(V_\varepsilon*\rho_0)(x)|^2\,dx\\
    &=\frac{1}{2}\|V_\varepsilon*\rho_0\|_{L^2(\Rd)}^2\le\frac{1}{2}\|V_\varepsilon\|_{L^1}^2\|\rho_0\|_{L^2}^2=\frac{1}{2}\|V_1\|_{L^1}^2\|\rho_0\|_{L^2}^2<\infty.
\end{align*}
\end{rem}

We now proceed with the JKO scheme. First, we define a sequence recursively as follows:
\begin{itemize}
    \item fix a time step $\tau>0$ such that $\rho_{\tau,\varepsilon}^0:=\rho_0$;
    \item for a given $\rhotne\in\mptrd$, choose
\begin{equation}\label{eq:jko}
    \rhotnne\in\argmin_{\rho\in\mptrd}\left\{\frac{d_W^2(\rhotne,\rho)}{2\tau}+\mw_\varepsilon[\rho]\right\}.
\end{equation}
\end{itemize}
The above sequence is well-defined if, for fixed $\bar\rho\in\mptrd$, the penalised energy functional $\rho\in\mptrd\mapsto\frac{d_W^2(\bar{\rho},\rho)}{2\tau}+\mw_\varepsilon[\rho]$ admits minimisers. This can be easily proven by applying the direct method of calculus of variations. For further details we refer to \cite[Lemma 2.3 and Proposition 2.5]{CDFFLS}, although we notice that in our case lower semi-continuity is easier. More precisely, the penalised energy functional is bounded from below and lower semicontinuous w.r.t. the narrow convergence by noticing that $W_\varepsilon$ is continuous and bounded from below, and the $2$-Wasserstein distance is lower semicontinuous. 

Let $T>0$ be fixed, and define a piecewise constant interpolation as follows: assume $N:=\left[\frac{T}{\tau}\right]$ and set
$$
\rhote(t)=\rhotne \qquad t\in((n-1)\tau,n\tau],
$$
being $\rhotne$ defined in \eqref{eq:jko}.

In the next Proposition we prove narrow compactness (in $\tau$) for $\rhote$ and two crucial estimates for its limiting curve, which we shall see it is a solution to \eqref{eq:nlie}. More precisely we prove uniform bounds in $\tau$ and $\varepsilon$ for the interaction energy and second order moments. 
\begin{prop}[Narrow compactness, energy $\&$ moments bound]\label{prop:en-ineq-mom-bound}
There exists an absolutely continuous curve $\tilde{\rho}^\varepsilon: [0,T]\rightarrow\mptrd$ such that the piecewise constant interpolation $\rhote$ admits a subsequence $\rho_{\tau_k}^\varepsilon$ narrowly converging to $\tilde{\rho}^\varepsilon$ uniformly in $t\in[0,T]$ as $k\rightarrow +\infty$. Moreover, for any $t\in[0,T]$, the following uniform bounds in $\tau$ and $\varepsilon$ hold
\begin{subequations}
\begin{align}
    \mw_\varepsilon[\tilde{\rho}^\varepsilon(t)]&\le\frac{1}{2}\|V_1\|_{L^1}^2\|\rho_0\|_{L^2}^2,\\
    m_2(\tilde{\rho}^\varepsilon)&\le2m_2(\rho_0)+2T\|V_1\|_{L^1}^2\|\rho_0\|_{L^2}^2.
\end{align}
\end{subequations}
\end{prop}
\begin{proof}
From the Definition of the sequence $\{\rhotne\}_{n\in\mathbb{N}}$ it holds
\begin{align}\label{eq:basic-ineq}
    \frac{d_W^2(\rhotne,\rhotnne)}{2\tau}+\mw_\varepsilon[\rhotnne]\le\mw_\varepsilon[\rhotne],
\end{align}
which implies $\mw_\varepsilon[\rhotnne]\le\mw_\varepsilon[\rhotne]$, and, in particular, the following bound for the interaction energy (cf. Remark \ref{rem:uniform-bound-initial-energy})
\begin{align}\label{eq:energy-ineq}
\sup_n\mw_\varepsilon[\rhotne]\le\mw_\varepsilon[\rho_0]\le\frac{1}{2}\|V_1\|_{L^1}^2\|\rho_0\|_{L^2}^2.
\end{align}
By summing up over $k$ inequality \eqref{eq:basic-ineq}, we obtain
\begin{align}
\sum_{k=m}^n\frac{d_W^2(\rho_{\tau,\varepsilon}^k,\rho_{\tau,\varepsilon}^{k+1})}{2\tau}\le \mw_\varepsilon[\rho_{\tau,\varepsilon}^{m}]-\mw_\varepsilon[\rhotnne].
\end{align}
The non-negativity of $W_\varepsilon$ and the energy inequality \eqref{eq:energy-ineq} allow us to improve the above inequality to
\begin{align}\label{eq:total-square-en-estim}
    \sum_{k=m}^n\frac{d_W^2(\rho_{\tau,\varepsilon}^k,\rho_{\tau,\varepsilon}^{k+1})}{2\tau}\le\mw_\varepsilon[\rho_0].
\end{align}
This implies
\begin{align*}
d_W^2(\rho_0,\rhote(t))\le 2T \mw_\varepsilon[\rho_0]\le T\|V_1\|_{L^1}^2\|\rho_0\|_{L^2}^2,
\end{align*}
whence we obtain that second order moments are uniformly bounded on $[0,T]$ in view of Remark \ref{rem:mom-ineq}, i.e. 
\begin{equation}\label{eq:mom-inequality}
    m_2(\rhote(t))\le 2m_2(\rho_0)+2d_W^2(\rho_0,\rhote(t))
    \le 2m_2(\rho_0)+2T\|V_1\|_{L^1}^2\|\rho_0\|_{L^2}^2.
\end{equation}
Now, let us consider $0\le s<t$ such that $s\in((m-1)\tau,m\tau]$ and $t\in((n-1)\tau,n\tau]$ (which implies $|n-m|<\frac{|t-s|}{\tau}+1$); by Cauchy-Schwarz inequality and \eqref{eq:total-square-en-estim}, we obtain
\begin{equation}\label{eq:holder-cont}
\begin{split}
d_W(\rhote(s),\rhote(t))&\le\sum_{k=m}^{n-1}d_W(\rho_{\tau,\varepsilon}^{k},\rho_{\tau,\varepsilon}^{k+1})\le\left(\sum_{k=m}^{n-1}d_W^2(\rho_{\tau,\varepsilon}^{k},\rho_{\tau,\varepsilon}^{k+1})\right)^{\frac{1}{2}}|n-m|^{\frac{1}{2}}\\&\le c \left(\sqrt{|t-s|}+\sqrt{\tau}\right),
\end{split}
\end{equation}\noindent
where $c$ is a positive constant.
Thus $\rhote$ is $\frac{1}{2}$-H\"{o}lder equi-continuous, up to a negligible error of order $\sqrt{\tau}$. By using a refined version of Ascoli-Arzel\`{a}'s theorem, \cite[Proposition 3.3.1]{AGS}, we obtain $\rhote$ admits a subsequence narrowly converging to a limit $\tilde{\rho}^\varepsilon$ as $\tau\to0^+$ uniformly on $[0,T]$.
Since $|\cdot|^2$ and $W_\varepsilon$ are lower semicontinuous and bounded from below, we actually have for any $t\in[0,T]$
\begin{align*}
    &\liminf_{k\to+\infty}\int_\Rd|x|^2\,d\rho_{\tau_k}^\varepsilon(x)\ge\int_\Rd|x|^2\,d\tilde{\rho}^\varepsilon(x)\\
    &\liminf_{k\to+\infty}\mw[\rho_{\tau_k}^\varepsilon]\ge\mw_\varepsilon[\tilde{\rho}^\varepsilon],
\end{align*}
whence the assertion follows.
\end{proof}
Next, we show that $\tilde{\rho}^\varepsilon$ provided by Proposition \ref{prop:en-ineq-mom-bound} is indeed a solution to \eqref{eq:nlie}. We stress that this result is not surprising and it is not the main purpose of this paper. Nevertheless, our interaction potential $W_\varepsilon$ does not satisfy the convexity assumption required in \cite{AGS,CDFFLS}, where there is a rigorous theory for weak measure solutions to \eqref{eq:nlie}. Therefore, for the sake of completeness we show that the lack of convexity does not affect existence of solutions to \eqref{eq:nlie}. In fact, we can pass to the limit in the Euler-Lagrange equation associated to \eqref{eq:jko}. 
\begin{thm}\label{thm:existence-nlie}
The curve $\tilde{\rho}^\varepsilon$ is a weak measure solution to \eqref{eq:nlie} according to Definition \ref{def:weak-meas-sol}.
\end{thm}
\begin{proof}
Let us consider two consecutive elements of the sequence $\{\rhotne\}_{n\in\mathbb{N}}$ defined in \eqref{eq:jko}, i.e. $\rhotne$ and $\rhotnne$. We perturb $\rhotnne$ by using a map $P^\sigma=id+\sigma\zeta$, for some $\zeta\in C_c^\infty(\Rd;\Rd)$ and $\sigma\ge0$, that is we consider the perturbation
\begin{equation}\label{eq:perturbation}
 \rho^\sigma:= P_\#^\sigma\rhotnne.
 \end{equation}
Being $\rhotnne$ a minimiser of \eqref{eq:jko}, we have 
\begin{equation}\label{eq:optimality}
\frac{1}{2\tau}\left[\frac{d_W^2(\rhotne,\rho^\sigma)- d_W^2(\rhotne, \rhotnne)}{\sigma}\right]+\frac{\mw_\varepsilon[\rho^\sigma]-\mw_\varepsilon[\rhotnne]}{\sigma}\ge0.
\end{equation}
First, we consider the interaction terms in \eqref{eq:optimality} 
\begin{equation}\label{eq:interaction-terms}
\begin{split}
&\frac{1}{2\sigma}\int_\Rd W_\varepsilon*\rho^\sigma(x) d\rho^\sigma(x)-\frac{1}{2\sigma}\int_{\Rd}W_\varepsilon *\rhotnne(x) d\rhotnne(x)\\
&=\frac{1}{2}\iint_\Rdd\left[\frac{W_\varepsilon(P^\sigma(x)-P^\sigma(y))-W_\varepsilon(x-y)}{\sigma}\right]\,d\rhotnne(y)\,d\rhotnne(x)
\\&=\frac{1}{2}\iint_{\Rdd}\left[\frac{W_\varepsilon(x-y +\sigma(\zeta(x)-\zeta(y)))-W_\varepsilon(x-y)}{\sigma}\right]\,d\rhotnne(y)\,d\rhotnne(x).
\end{split}
\end{equation}
Since the interaction potential satisfies $W_\varepsilon\in C(\Rd)\cap C^1(\Rd\setminus\{0\})$, for all $(x,y)\in \Rdd$ it holds
\begin{equation}\label{eq:converg-W}
\frac{W_\varepsilon(x-y +\sigma(\zeta(x)-\zeta(y)))-W_\varepsilon(x-y)}{\sigma} \underset{\sigma\to 0}{\longrightarrow}\nabla W_\varepsilon(x-y)\cdot(\zeta(x)-\zeta(y)).
\end{equation}
By means of Egorov's theorem, for every $\eta>0$ there exists $B_\eta\subset\Rdd$ measurable such that $$\iint_{B_\eta}\,d\rhotnne(y)\,d\rhotnne(x)<\eta$$ and the convergence \eqref{eq:converg-W} is uniform on $\Rdd\setminus B_{\eta}$. The integral on $B_\eta$ can be neglected in the limit-integral interchange since the sequence in \eqref{eq:converg-W} is uniformly bounded in $\sigma$. Thus, we obtain
\begin{equation*}
\begin{split}
&\frac{1}{2}\iint_{\Rdd}\left(\frac{W_\varepsilon(x-y +\sigma(\zeta(x)-\zeta(y)))-W_\varepsilon(x-y)}{\sigma}\right)\,d\rhotnne(y)\,d\rhotnne(x)\\
&\quad\underset{\sigma\to0}{\longrightarrow} \frac{1}{2}\iint_{\Rdd}\nabla W_\varepsilon(x-y)\cdot(\zeta(x)-\zeta(y))\,d\rhotnne(y)\,d\rhotnne(x).
\end{split}
\end{equation*}
Regarding the terms involving the $2$-Wasserstein distance, let us consider an optimal transport plan $\gamma_{\tau,\varepsilon}^n\in\Gamma_0(\rhotne,\rhotnne)$ between $\rhotne$ and $\rhotnne$. By Definition of $d_W$, we have
\begin{equation*}
\begin{split}
\frac{1}{2\tau}\left[\frac{d_W^2(\rhotne, \rho^\sigma)-d_W^2(\rhotne, \rhotnne)}{\sigma}\right]&\le\frac{1}{2\tau\sigma}\iint_\Rdd\left(|x-P^\sigma(y)|^2 -|x-y|^2\right)\,d\gamma_{\tau,\varepsilon}^n(x,y)\\
&=\frac{1}{2\tau\sigma}\iint_\Rdd\left(|x-y-\sigma\zeta(y)|^2 -|x-y|^2\right)\,d\gamma_{\tau,\varepsilon}^n(x,y)\\
&=-\frac{1}{\tau}\iint_\Rdd(x-y)\cdot \zeta(y)\,d\gamma_{\tau,\varepsilon}^n(x,y)+o(\sigma),
\end{split}
\end{equation*}
where in the last equality we applied a first order Taylor expansion (note that $\zeta \in C^\infty$). By sending $\sigma$ to $0$ it holds
\begin{equation*}
    \frac{1}{\tau}\iint_\Rdd(x-y)\cdot \zeta(y)\,d\gamma_{\tau,\varepsilon}^n(x,y)\le\frac{1}{2}\iint_{\Rdd}\nabla W_\varepsilon(x-y)\cdot(\zeta(x)-\zeta(y))\,d\rhotnne(y)\,d\rhotnne(x).
\end{equation*}
Repeating the same computation for $\sigma\le0$, we actually obtain an equality, that is, for $\zeta=\nabla\varphi$
\begin{equation}\label{eq:sol-discreta-sigma}
\begin{split}
\frac{1}{\tau}\!\iint_\Rdd(x-y)\cdot \nabla\varphi(y)d\gamma_{\tau,\varepsilon}^n(x,y)\!=\!\frac{1}{2}\!\iint_{\Rdd}\nabla W_\varepsilon(x-y)\cdot(\nabla\varphi(x)\!-\!\nabla\varphi(y))d\rhotnne(y)d\rhotnne(x).
\end{split}
\end{equation}
Note that the H\"older estimate \eqref{eq:holder-cont} and $(x-y)\cdot \nabla\varphi(y)=\varphi(x)-\varphi(y)+o(|x-y|^2)$ imply
\[
\frac{1}{\tau}\iint_\Rdd(x-y)\cdot \nabla\varphi(y)\,d\gamma_{\tau,\varepsilon}^n(x,y)=\frac{1}{\tau}\int_\Rd\varphi(x)\,d(\rhotne-\rhotnne)(x) + O(\tau).
\]
Now, let $0\le s<t$ be fixed, with
$$
h=\left[\frac{s}{\tau}\right]+1\quad \text{and}\quad k=\left[\frac{t}{\tau}\right].
$$
Taking into account the last equality, by summing in \eqref{eq:sol-discreta-sigma} over $j$ from $h$ to $k$, we obtain
\begin{align*}
    \int_\Rd\varphi(x)\,d\rho_{\tau,\varepsilon}^{k+1}-&\int_\Rd\varphi(x)\,d\rho_{\tau,\varepsilon}^h+O(\tau^2)=\\
    &-\sum_{j=h}^k\frac{\tau}{2}\iint_{\Rdd}\nabla W_\varepsilon(x-y)\cdot(\nabla\varphi(x)\!-\!\nabla\varphi(y))\,d\rho_{\tau,\varepsilon}^{j+1}(y)\,d\rho_{\tau,\varepsilon}^{j+1}(x),
\end{align*}
which is equivalent to
\begin{align*}
    \int_\Rd\varphi(x)\,d\rhote(t)(x)-&\int_\Rd\varphi(x)\,d\rhote(s)(x)+O(\tau^2)=\\
    &-\frac{1}{2}\int_s^t\iint_{\Rdd}\nabla W_\varepsilon(x-y)\cdot(\nabla\varphi(x)\!-\!\nabla\varphi(y))\,d\rhote(r)(y)\,d\rhote(r)(x)\,dr.
\end{align*}
Up to pass to a subsequence, the result follows by considering the limit as $\tau\to0^+$ and choosing $s=0$.
\end{proof}

\section{Compactness for $\rho^\varepsilon$ and $v^\varepsilon$}\label{sec:compactness}

The sequence of solutions $\{\tilde{\rho}^\varepsilon\}_{\varepsilon>0}$ to \eqref{eq:nlie} constructed in section \ref{sec:nlie} is the candidate approximating \textit{weak solution} of \eqref{eq:pme2}, if we use higher regularity of its smoothed version, $V_\varepsilon*\tilde{\rho}^\varepsilon$, in the  limit $\varepsilon\to0^+$. In this section we deal with the compactness for both the sequences. For the ease of presentation, from this point on we drop the symbol tilde used in the previous section to denote the sequence of solutions to \eqref{eq:nlie} in Theorem \ref{thm:existence-nlie}.

First, we prove that $\{\rhoe\}_{\varepsilon>0}$ is relatively compact in $C([0,T],\mptrd)$, again by means of a refined version of the Ascoli-Arzelà theorem, \cite[Proposition 3.3.1]{AGS}.
\begin{prop}\label{prop:limit-rho}
There exists an absolutely continuous curve $\tilde{\rho}:[0,T]\to\mptrd$ such that the sequence $\{\rho^\varepsilon\}_{\varepsilon>0}$ admits a subsequence $\{\rho^{\varepsilon_k}\}$ such that $\rho^{\varepsilon_k}(t)$ narrow converges to $\tilde{\rho}(t)$ for any $t\in[0,T]$ as $k\to+\infty$.
\end{prop}
\begin{proof}
Firstly, a subset $K\subset\mptrd$ is relatively compact if and only if it is tight, due to Prokhorov's theorem. The sequence $\rho^\varepsilon$ is tight since its second order moments are uniformly bounded according to Proposition \ref{prop:en-ineq-mom-bound}.
Secondly, the equi-continuity of $\rho^\varepsilon$ follows from that of $\rhote$, cf.~ \eqref{eq:holder-cont}, by lower semi-continuity of the $2$-Wasserstein distance. More precisely, for any $\varepsilon>0$ and $s,t\in[0,T]$, let us consider a sequence of optimal transport plans $\gamma_\tau^\varepsilon\in\Gamma_0(\rhote(s),\rhote(t))$ such that
\[
d_W^2(\rhote(s),\rhote(t))=\iint_{\Rdd}|x-y|^2\,d\gamma_\tau^\varepsilon(x,y).
\]
By stability of optimal transport plans, cf.~\cite[Corollary 5.21]{V2} , we get $\gamma_\tau^\varepsilon\rightharpoonup\gamma^\varepsilon$ as $\tau\to0^+$, and
\begin{align*}
    \liminf_{\tau\to0}d_W^2(\rhote(s),\rhote(t))&=\liminf_{\tau\to0}\iint_{\Rdd}|x-y|^2\,d\gamma_\tau^\varepsilon(x,y)\\
    &\ge\iint_{\Rdd}|x-y|^2\,d\gamma^\varepsilon(x,y)\\
    &\ge d_W^2(\rhoe(s),\rhoe(t)).
\end{align*}
In particular, from \eqref{eq:holder-cont} in Proposition \ref{prop:en-ineq-mom-bound} we obtain
\[
d_W(\rhoe(s),\rhoe(t))\le c|t-s|,
\]
for a positive constant $c$. Finally, the assertion follows by applying the aforementioned version of the Ascoli-Arzelà theorem.
\end{proof}

Next we consider the corresponding smoothed (sub)sequence $\{v^\varepsilon\}_\varepsilon$, being $v^\varepsilon(t):=V_\varepsilon*\rhoe(t)$ for any $t\in[0,T]$. Note that we removed the subscript $k$ for ease of presentation. For the latter sequence we obtain an $H^1$ estimate by using the \textit{flow interchange technique}, developed by Matthes, McCann and Savar\'{e} in \cite{MMCS}, cf. also \cite{DFM,DiFranEspFag,CDFEFS20} for further details. The strategy is to compute the dissipation of the interaction energy functional $\mw_\varepsilon$ along a solution of an \textit{auxiliary gradient flow}, in order to use the \textit{Evolution Variational Inequality (EVI)} to obtain the desired estimate, leading to compactness.

Since the seminal work by Jordan, Kinderlehrer, and Otto, \cite{JKO},  it is known that the heat equation can be regarded as a $2$-Wasserstein steepest descent of the opposite of the Boltzmann entropy, i.e. $\mh[\rho]=\int_{\Rd}\rho(x)\log\rho(x)\,dx$. The entropy functional is $0$-convex along geodesics and it possesses a unique $0$-flow, denoted by $S_\mh$, given by the heat semigroup (cf.~\cite{AGS,DS,DFM}). For the reader's convenience we recall the Definition of $\lambda$-flow for a general functional $\mathcal{F}$.
\begin{defn}[$\lambda$-flow]
A semigroup $S_{\mathcal{F}}:[0,+\infty]\times\mptrd\to\mptrd$ is a $\lambda$-flow for a functional $\mathcal{F}:\mptrd\to\R\cup\{+\infty\}$ with respect to the distance $d_W$ if, for an arbitrary $\rho\in\mptrd$, the curve $t\mapsto S_{\mathcal{F}}^t\rho$ is absolutely continuous on $[0,+\infty[$ and it satisfies the evolution variational inequality (EVI)
\begin{equation}
\frac{1}{2}\frac{d^+}{dt}d_W^2(S_{\mathcal{F}}^t\rho,\bar{\rho})+\frac{\lambda}{2}d_W^2(S_{\mathcal{F}}^t\rho,\bar{\rho})\le \mathcal{F}(\bar{\rho})-\mathcal{F}(S_{\mathcal{F}}^t\rho)
\end{equation}
for all $t>0$, with respect to every reference measure $\bar{\rho}\in\mptrd$ such that $\mathcal{F}(\bar{\rho})<\infty$.
\end{defn}
Below we use the flow interchange by considering the heat equation as auxiliary flow, and the entropy as auxiliary functional, i.e.
\begin{equation}\label{eq:aux-func}
\mh[\rho]=
\begin{cases}
\int_{\Rd}\rho(x)\log\rho(x)\,dx, &\rho\log\rho\in L^1(\Rd);\\
+\infty & \text{otherwise}.
\end{cases}
\end{equation}
\begin{rem}\label{eq:controlbelowentropy}
We remind the reader that the entropy is controlled from below by the second order moment of $\rho$, denoted by $m_2(\rho)$. More precisely, in~\cite[Proposition 4.1]{JKO} it is shown that
$$
\mh(\rho)\ge -C(m_2(\rho) + 1)^{\beta},
$$
for every $\rho\in\mpdtard$, $\beta\in(\frac{d}{d+2},1)$ and $C<+\infty$, depending only on the space dimension $d$. We use this bound in order to have a uniform control from below for the entropy.
\end{rem}
In the following, for any $\nu\in\mptrd$ such that $\mh(\nu)<+\infty$, we denote by $S_{\mh}^t\nu$ the solution at time $t$ of the heat equation coupled with an initial value $\nu$ at $t=0$. Moreover, for every $\rho\in\mptrd$, we define the dissipation of $\mw_\varepsilon$ along $S_{\mh}$ by
$$
D_{\mh}\mw_\varepsilon(\rho):=\limsup_{s\downarrow0}\left\{\frac{\mw_\varepsilon[\rho]-\mw_\varepsilon[S_{\mh}^s\rho]}{s}\right\}.
$$
We can now prove a uniform bound for $\{v^\varepsilon\}_\varepsilon$ in $L^2([0,T];H^1(\Rd))$.
\begin{lem}\label{lem:h1-bound}
Let $\rho_0\in\mpdtard\cap L^2(\Rd)$ such that $\mh[\rho_0]<\infty$. There exists a constant $C=C(\rho_0,V_1,T)$ such that, for any $\varepsilon>0$,
\begin{align}
    \|v^\varepsilon\|_{L^2([0,T];H^1(\Rd))}\le C(\rho_0,V_1,T).
\end{align}
Therefore, there exists a subsequence $\{v^{\varepsilon_k}\}_k$ and a curve $v\in L^2([0,T];H^1(\Rd))$ such that $v^{\varepsilon_k}\rightharpoonup v$ in $L^2([0,T];H^1(\Rd))$.
\end{lem}
\begin{proof}
From Proposition \ref{prop:en-ineq-mom-bound} we infer the uniform bound in $\tau$ and $\varepsilon$
\begin{equation*}
\begin{split}
    \|V_\varepsilon*\rhote\|_{L^2([0,T];L^2(\Rd))}^2&=\int_0^T\int_\Rd|[V_\varepsilon*\rhote(t)](x)|^2\,dx\,dt=2\int_0^T\mw_\varepsilon[\rhote(t)]\,dt\\
    &\le2T\mw_\varepsilon[\rho_0]\le T\|V_1\|_{L^1(\Rd)}^2\|\rho_0\|_{L^2(\Rd)}^2.
\end{split}
\end{equation*}
Thus, there exists a subsequence $\tau_k$ such that $V_\varepsilon*\rhoe_{\tau_k}\rightharpoonup w^\varepsilon$ in $L^2([0,T];L^2(\Rd))$ as $\tau_k\to0$. The limit $w^\varepsilon\equiv v^\varepsilon$ due to uniqueness of limit and Proposition \ref{prop:limit-rho}. Up to pass to a subsequence, we have 
\begin{equation}\label{eq:l2-boundveps}
\|v^\varepsilon\|_{L^2([0,T];L^2(\Rd))}^2\le T\|V_1\|_{L^1(\Rd)}^2\|\rho_0\|_{L^2(\Rd)}^2,
\end{equation}
since the norm is weakly lower semicontinuous.
Now, we obtain a uniform bound for $\nabla v^\varepsilon$. For all $s>0$, if we consider $S_\mh^s\rhotnne$ as competitor of $\rhotnne$ in the minimisation problem \eqref{eq:jko}, as direct consequence of the Definition of the sequence $\{\rhotne\}_{n\in\mathbb{N}}$ we have
$$
\frac{1}{2\tau}d_W^2(\rhotnne,\rhotne)+\mw_\varepsilon[\rhotnne]\le\frac{1}{2\tau}d_W^2(S_{\mh}^s\rhotnne,\rhotne)+\mw_\varepsilon[S_{\mh}^s\rhotnne],
$$
whence, dividing by $s>0$ and passing to the $\limsup$ as $s\downarrow0$,
\begin{equation}\label{eq:flow-interchange}
\tau D_{\mh}\mw_\varepsilon(\rhotnne)\le\frac{1}{2}\frac{d^+}{dt}\Big(d_W^2(S_{\mh}^t\rhotnne,\rhotne)\Big)\big|_{t=0}\overset{\bm{(E.V.I.)}}{\le}\mh[\rhotne]-\mh[\rhotnne].
\end{equation}
In the last inequality we used that $S_\mh$ is a $0$-flow. Now, let us focus on the left hand side of \eqref{eq:flow-interchange}. First of all, note that
\begin{equation}\label{eq:integral-form-dis}
\begin{split}
D_{\mh}\mw_\varepsilon(\rhotnne)&=\limsup_{s\downarrow0}\left\{\frac{\mw_\varepsilon[\rhotnne]-\mw_\varepsilon[S_{\mh}^s\rhotnne]}{s}\right\}\\&=\limsup_{s\downarrow0}\int_0^1\left(-\frac{d}{dz}\Big|_{z=st}\mw_\varepsilon[S_{\mh}^{z}\rhotnne]\right)\,dt.
\end{split}
\end{equation}
Thus, we now compute the time derivative inside the above integral, by using integration by parts and keeping in mind the regularity of the solution to the heat equation:
\begin{equation}\label{eq:deriv-dis}
\begin{split}
\frac{d}{dt}\mw_\varepsilon[S_{\mh}^t\rhotnne]=&-\int_{\Rd}\nabla (W_\varepsilon*S_{\mh}^t\rhotnne)(x)\nabla S_{\mh}^t\rhotnne(x)\,dx\\
&=-\int_{\Rd}|\nabla( V_\varepsilon*S_{\mh}^t\rhotnne)(x)|^2\,dx.
\end{split}
\end{equation}
By substituting \eqref{eq:deriv-dis} into \eqref{eq:integral-form-dis}, from \eqref{eq:flow-interchange} we obtain
\[
\tau\liminf_{s\downarrow0}\int_0^1\int_{\Rd}|\nabla (V_\varepsilon*S_{\mh}^{st}\rhotnne)(x)|^2\,dx\,dt\le\mh[\rhotne]-\mh[\rhotnne],
\]
whence, by $L^2$ lower semi-continuity of the $H^1$ seminorm,
\[
\tau\int_{\Rd}|\nabla (V_\varepsilon*\rhotnne)(x)|^2\,dx\,dt\le\mh[\rhotne]-\mh[\rhotnne].
\]
By summing up over $n$ from $0$ to $N-1$, taking into account that $x\log x\le x^2$ for any $x\ge0$, Remark \ref{eq:controlbelowentropy} and that second order moments are uniformly bounded (see Proposition \ref{prop:en-ineq-mom-bound}), we get
\begin{align*}
\int_0^T\int_{\Rd}|\nabla( V_\varepsilon*\rhote(t))(x)|^2\,dx\,dt\le \mh[\rho_0]-\mh[\rhotne]\le \|\rho_0\|_{L^2(\Rd)}^2+C(\rho_0,V_1,T).
\end{align*}
In particular, using weak lower semi-continuity of the norm,
\begin{equation}\label{eq:bound-nablaveps}
\|\nabla v^\varepsilon\|_{L^2([0,T];L^2(\Rd))}^2=\int_0^T\int_{\Rd}|\nabla v^\varepsilon(t)(x)|^2\,dx\,dt\le \|\rho_0\|_{L^2(\Rd)}^2+C(\rho_0,V_1,T).
\end{equation}
The bounds in  \eqref{eq:l2-boundveps} and \eqref{eq:bound-nablaveps} give the first result of the statement, and an application of the Banach–Alaoglu Theorem concludes the proof.
\end{proof}
\begin{rem}
We observe that $\rho_0\in\mpdtard\cap L^2(\Rd)$ does imply $\mh[\rho_0]<\infty$, since $x\log x\le x^2$ for any $x\ge0$. We prefer to keep this assumption as this is the main issue when dealing with the particle approximation using this strategy.
\end{rem}
\begin{rem}\label{rem:moment-veps}
Let us notice that, for any $t\in[0,T]$, the first order moment of $v^\varepsilon_t$ is finite. In fact, using that $V_\varepsilon$ is even, the assumption $\int_\Rd|x|V_1(x)\,dx<+\infty$, and Proposition \ref{prop:en-ineq-mom-bound}:
\begin{align*}
\int_\Rd|x|v^\varepsilon_t(x)\,dx&=\iint_\Rdd |x|V_\varepsilon(x-y)\,d\rho^\varepsilon_t(y)\,dx\\
&\le \iint_\Rdd V_\varepsilon(y-x)|x-y|\,d\rho^\varepsilon_t(y)\,dx+ \iint_\Rdd V_\varepsilon(y-x)|y|\,d\rho^\varepsilon_t(y)\,dx\\
&=\varepsilon\int_\Rd V_1(z)|z|\,dz+\int_\Rd V_1(z)\,dz\int_\Rd|y|\,d\rho^\varepsilon_t(y)\\
&\le\varepsilon\int_\Rd V_1(z)|z|\,dz+\sqrt{m_2(\rhoe)}\int_\Rd V_1(z)\,dz<+\infty.
\end{align*}
\end{rem}
The strong $L^2$ compactness in time and space follows by applying a refined version of the Aubin-Lions Lemma due to Rossi and Savar\'{e} \cite[Theorem 2]{RS}. For the reader's convenience we recall the latter result below, before presenting the compactness result for $\{v^{\varepsilon_k}\}_k$.

\begin{prop}\cite[Theorem 2]{RS}\label{prop:aulirs-meas}
Let $X$ be a separable Banach space. Consider
\begin{itemize}
\item a lower semicontinuous functional $\mathcal{F}:X\to[0,+\infty]$ with relatively compact sublevels in $X$;
\item a pseudo-distance $g:X\times X\to[0,+\infty]$, i.e., $g$ is lower semicontinuous and such that $g(\rho,\eta)=0$ for any $\rho,\eta\in X$ with $\mathcal{F}(\rho)<\infty$, $\mathcal{F}(\eta)<\infty$ implies $\rho=\eta$.
\end{itemize}
Let $U$ be a set of measurable functions $u:(0,T)\to X$, with a fixed $T>0$. Assume further that
\begin{equation}\label{hprossav}
\sup_{u\in U}\int_{0}^T\mathcal{F}(u(t))\,dt<\infty\quad \text{and}\quad \lim_{h\downarrow0}\sup_{u\in U}\int_{0}^{T-h}g(u(t+h),u(t))\,dt=0\,.
\end{equation}
Then $U$ contains an infinite sequence $(u_n)_{n\in\mathbb{N}}$ that converges in measure, with respect to $t\in(0,T)$, to a measurable $\tilde{u}:(0,T)\to X$, i.e.
\[
\lim_{n\to\infty}|\{t\in(0,T):\|u_n(t)-u(t)\|_X\ge\sigma\}|=0, \quad \forall \sigma>0.
\]
\end{prop}

The two conditions in \eqref{hprossav} are called \textit{tightness} and \textit{weak integral equicontinuity}, respectively.
\begin{prop}\label{prop:strong-convergence-v}
Let $\varepsilon\le1$. The sequence $\{v^{\varepsilon_k}\}_k$ obtained in Lemma \ref{lem:h1-bound} converges strongly to the curve $v$ in $L^2([0,T];L^2(\Rd))$, for any $T>0$.
\end{prop}
\begin{proof}
The proof of the result is obtained by applying Proposition \ref{prop:aulirs-meas} to a subset of $U:=\{v^\varepsilon\}_{0\le\varepsilon\le1}$ for $X:=L^2(\Rd)$ and $g:=d_1$ being the $1$-Wasserstein distance --- extended to $+\infty$ outside of $\mP_1(\Rd)\times\mP_1(\Rd)$. As for the functional, we consider $\mathcal{F}:L^2(\Rd)\to[0,+\infty]$ defined by
\begin{equation*}
\mathcal{F}[v]=
\begin{cases}
||v||_{H^1(\Rd)}^2 + \int_{\R^d}|x|v(x)\, dx, & \text{if } v\in\mP_1(\Rd)\cap H^1(\Rd);\\
+\infty & \text{otherwise}.
\end{cases}
\end{equation*}
Note that elements in the domain of the functional $\mathcal{F}$ belong to $\mP_1(\Rd)$, thus $0=g(\rho,\eta)=d_1(\rho,\eta)$ implies $\rho=\eta$. Next we show that $\mathcal{F}$ is an admissible functional and later on we check the conditions in \eqref{hprossav}. In order to improve the readability we split the remainder of the proof in four steps.

\textit{\underline{Step 1: $\mathcal{F}$ is lower semicontinuous}} Let $\{v_n\}_n\subset L^2(\Rd)$ such that $v_n\to v$ in $L^2(\Rd)$ and $F[v_n]<+\infty$, otherwise it is trivial. We prove that $\mathcal{F}^1[v]:=||\nabla v||_{L^2(\Rd)}^2$ and $\mathcal{F}^2[v]:=\int_{\R^d}|x|v(x)\, dx$ are lower semicontinuous, since $\|v\|_{L^2}^2(\Rd)$ obviously is.
Note that $\|v_n\|_{H^1(\Rd)}^2\le\sup_n\|v_n\|_{H^1(\Rd)}^2=:\bar{\mathcal{F}}<+\infty$. Thus, there exists a subsequence such that $\nabla v_{n_k}\rightharpoonup \nabla v$ in $L^2(\Rd)$, since the limit is unique. A straightfoward computation shows that
\begin{align*}
    \mathcal{F}^1[v_n]\ge\int_{\Rd}|\nabla v(x)|^2\,dx+2\int_{\Rd}(\nabla v_n(x)-\nabla v)\cdot\nabla v(x)\,dx,
\end{align*}
which gives $\liminf_n\mathcal{F}^1[v_n]\ge\mathcal{F}^1[v]$. Regarding $\mathcal{F}^2$, let us consider $B_R$ a ball of radius $R$. Since $v_n\rightharpoonup v$ in $L^2(\Rd)$ and $|\cdot|\in L^2(B_R)$, we have
\[
\lim_n\int_{B_R}|x|v_n(x)\,dx=\int_{B_R}|x|v(x)\,dx,
\]
whence
\[
\liminf_n\int_{\Rd}|x|v_n(x)\,dx\ge\liminf_n\int_{B_R}|x|v_n(x)\,dx=\int_{B_R}|x|v(x)\,dx
\]
The monotone convergence theorem gives the desired result.

\textit{\underline{Step 2: sublevels of $\mathcal{F}$ are relatively compact in $L^2(\Rd)$}} Let $A_c:=\{v\in L^2(\Rd): \mathcal{F}[v]\le c\}$ be a sublevel of $\mathcal{F}$, where $c$ is a positive constant. The Riesz-Fr\'{e}chet-Kolmogorov theorem provides relatively compactness in $L^2(\Rd)$ of $A_c$. In fact, elements of $A_c$ are bounded in $L^2(\Rd)$ and it holds the uniform continuity estimate
\begin{equation}\label{eq:l2-cont-est}
\begin{split}
\int_{\Rd}|v(x+h)-v(x)|^2\,dx &= \int_{\Rd}\left|\int_0^1 \frac{d}{d\tau}v(x+\tau h)\,d\tau\right|^2 \,dx = \int_{\Rd}\left|\int_0^1 h \cdot \nabla v(x+\tau h)\,d\tau\right|^2 \,dx\\
&\le |h|^2\int_{\Rd}\int_0^1 |\nabla v(x+\tau h)|^2\,d\tau \,dx = |h|^2\|\nabla v\|_{L^2(\Rd)}^2,
\end{split}
\end{equation}
which implies $\|v(\cdot+h)-v(\cdot)\|_{L^2(\Rd)}\to0$ as $h\to0^+$. Moreover, we have uniform integrability at infinity by means of H\"older and Gagliardo-Nirenberg inequalities. In particular,
\begin{align*}
\|v\|_{L^2(\Rd\setminus B_R)}^2&=\int_{|x|\ge R}|v(x)|^2 \,dx\le \frac{1}{R^{\delta}}\int_{\Rd}|x|^{\delta}|v(x)|^2\,dx\\ 
&\le \frac{1}{R^{\delta}}\left(\int_{\R^d}|x|v(x)\,dx\right)^\delta \left(\int_{\Rd}|v(x)|^{\frac{2-\delta}{1-\delta}}\,dx\right)^{1-\delta},
\end{align*}
where $\delta$ can be chosen in $(0,1)$ such a way the exponent $p:=(2-\delta)/(1-\delta)$ satisfies $p\in\left(2,+\infty\right)$ for $d=1,2$, and $2<p<\frac{2d}{d-2}$ for $d>2$. The latter requirements are implied by the Gagliardo-Nirenberg inequality
$$
\|v\|_{L^p(\Rd)}\le C\|\nabla v\|_{L^2(\Rd)}^{\theta}\|v\|_{L^2(\Rd)}^{1-\theta},\qquad \theta=\frac{(p-2)d}{2 p},
$$
which guarantees that $\|v\|_{L^p(\Rd)}$ is finite, thus the uniform integrability at infinity.

\textit{\underline{Step 3: tightness and weak integral equicontinuity}} Let us set $U:=\{v^\varepsilon\}_{0\le\varepsilon\le1}$, being $v^\varepsilon:[0,T]\to L^2(\Rd)$ the sequence defined above by $v^\varepsilon=V_\varepsilon*\rhoe$, which satisfies Lemma \ref{lem:h1-bound}. For any $0\le\varepsilon\le1$, it holds
\begin{align*}
    \int_0^T\mathcal{F}[v^\varepsilon(t)]\,dt&=\int_0^T||v^\varepsilon(t)||_{H^1(\Rd)}^2\, dt  + \int_0^T\int_{\R^d}|x|v^\varepsilon_t(x)\, dx\,dt\\
    &\le C(\rho_0,V_1,T)+ T\int_\Rd V_1(z)|z|\,dz<+\infty,
\end{align*}
where we also used Remark \ref{rem:moment-veps} and that $\varepsilon\le1$ --- note that the bound for $\varepsilon$ is arbitrary as we could choose any constant. Taking the supremum in $U$ we have tightness. The weak integral equi-continuity is a consequence of the equi-continuity of $\rhoe$ proven in Proposition \ref{prop:limit-rho}. More precisely, for any $\varepsilon\ge0$ and $h>0$ it holds
\begin{align*}
    \int_0^{T-h}\!\!d_1(v^\varepsilon(t+h),v^\varepsilon(t))\,dt\le\!\!\int_0^{T-h}\!\!d_W(v^\varepsilon(t+h),v^\varepsilon(t))\,dt\le\!\!\int_0^{T-h}\!\!d_W(\rhoe(t+h),\rhoe(t))\,dt\le c|h|T,
\end{align*}
where in the intermediate inequalities we used well known properties of Wasserstein distances, cf. for example \cite[Section 5.1]{S}.

\textit{\underline{Step 4: relatively compactness in $L^2([0,T];L^2(\Rd))$}} By abuse of notation we denote by $U:=\{v^{\varepsilon_k}\}_k$ a subsequence of $\{v^\varepsilon\}_{0\le\varepsilon\le1}$ such that $v^{\varepsilon_k}\rightharpoonup v$ in $L^2([0,T];H^1(\Rd))$, in view of Lemma \ref{lem:h1-bound}. According to Proposition \ref{prop:aulirs-meas}, there exists a subsubsequence $v^{\varepsilon_k'}$ such that $v^{\varepsilon_k'}$ converges in measure (with respect to time with values in $X=L^2(\Rd)$) to a curve $\tilde{v}\equiv v$, due to the weakly convergence of $v^{\varepsilon_k}$. By standard arguments we can conclude that $\{v^{\varepsilon_k}\}$ converges to $v$ in measure, thus pointwise almost everywhere (up to pass to a subsequence). Since $\sup_t\|v^\varepsilon(t)\|_{L^2(\Rd))}^2\le \|V_1\|_{L^1}^2\|\rho_0\|_{L^2}^2$, we infer strong convergence of $v^{\varepsilon_k}$ to $v$ in $L^2([0,T];L^2(\Rd))$ by applying Lebesgue's dominated convergence theorem.
\end{proof}

\section{Towards the quadratic porous medium equation}\label{sec:local-limit-pme2}

In view of the analysis carried out in the previous sections, we are now able to prove convergence of solutions of \eqref{eq:nlie} to the solution of \eqref{eq:pme2}, as $\varepsilon \rightarrow 0^+$. The key issue is to pass to the limit in the weak formulation, which is not straightforward since $\rho^\varepsilon$ is only a measure in general. As already explained earlier in the paper, we use the higher regularity of $v^\varepsilon$ and that $\rho^\varepsilon - v^\varepsilon$ converges to zero in the sense of distributions, starting from $\rho_0\in\mpdtard\cap L^2(\Rd)$ with $\mh[\rho_0]<\infty$.

According to Definition \ref{def:weak-meas-sol}, for any $\varepsilon>0$ and any $\varphi\in C^1_c(\Rd)$, $\rhoe$ satisfies 
\begin{equation}\label{eq:weak-form-conv}
\begin{split}
   \int_\Rd\varphi(x)d\rho_T^\varepsilon(x)\!-\!\int_\Rd\varphi(x)d\rho_0(x)&\!=\!-\frac{1}{2}\int_0^T\!\!\!\iint_\Rdd(\nabla\varphi(x)\!-\!\nabla\varphi(y))\!\cdot\!\nabla W_\varepsilon(x\!-\!y)d\rho_t^\varepsilon(y)d\rho_t^\varepsilon(x)dt,\\
    &=\!-\!\int_0^T\!\!\! \int_{\Rd} \nabla \varphi(x) \cdot \nabla V_\varepsilon*v^\varepsilon_t(x) d\rhoe_t(x)dt,
\end{split}
\end{equation}
which can be rewritten as
\begin{equation}\label{eq:weak-form-z}
\begin{split}
 \int_\Rd\varphi(x)\,d\rho_T^\varepsilon(x)-\int_\Rd\varphi(x)\,d\rho_0(x)&=-\int_0^T \int_{\Rd} V_\varepsilon* ( \rhoe_t \nabla \varphi)(x) \cdot \nabla v^\varepsilon_t(x) \,dx\,dt \\  &=-\int_0^T \int_{\Rd} v^\varepsilon_t(x) \nabla \varphi(x) \cdot \nabla v^\varepsilon_t(x) \,dx\,dt\\
&\quad -\int_0^T \int_{\Rd} z^\varepsilon_t(x) \cdot \nabla v^\varepsilon_t(x) \,dx\,dt,
\end{split}
\end{equation} 
where for any $t\in[0,T]$ and $x\in\Rd$ the excess term is given by
\[
z^\varepsilon_t(x) := V_\varepsilon* ( \rhoe_t \nabla \varphi)(x) - 
(V_\varepsilon*\rhoe_t)(x) \nabla \varphi(x) =  V_\varepsilon* ( \rhoe_t \nabla \varphi)(x) -v^\varepsilon_t(x) \nabla \varphi(x).
\]
\begin{rem}
Note that the integral after the second equality in \eqref{eq:weak-form-conv} makes sense since $\nabla V_\varepsilon*v_t^\varepsilon\in C(\Rd)$. This can be easily verified by applying Lebesgue dominated convergence theorem using that $\nabla V_\varepsilon\in L^1(\Rd)$, for $\varepsilon>0$.
\end{rem}
Lemma \ref{lem:h1-bound} and Proposition \ref{prop:strong-convergence-v} entail to pass to the limit in the first term on the right-hand side of \eqref{eq:weak-form-z}, upon considering a subsequence, since $v^\varepsilon$ converges strongly in $L^2([0,T];L^2(\Rd))$ and $\nabla v^\varepsilon$ converges weakly in $L^2([0,T];L^2(\Rd))$. We will show that $v^\varepsilon$ converges to the same limit of $\rho^\varepsilon$ in the sense of distributions, whence we infer that the limit $\tilde{\rho}$ from Proposition \ref{prop:limit-rho} attains the same regularity, namely $L^2([0,T];H^1(\Rd))$. Furthermore, we prove that the excess term $z^\varepsilon$ converges to $0$ in $L^2([0,T]\times\Rd)$.

\subsection{Convergence of the excess term}\label{sec:excess-term}

\begin{lem}
The excess term $z^\varepsilon$ satisfies
\begin{align*}
    \Vert z^\varepsilon \Vert_{L^\infty([0,T];L^1(\Rd))} \leq \varepsilon C(V_1,\varphi),
\end{align*}
for any $\varphi\in C^2_c(\Rd)$.
\end{lem}
\begin{proof}
For any $t\in[0,T]$ and $\varphi\in C^2_c(\Rd)$ we obtain
\begin{align*}
    \int_\Rd |z^\varepsilon_t(x)|\,dx &\le \int_{\Rd}\int_{\Rd} V_\varepsilon(x-y) |\nabla \varphi(y) - \nabla \varphi(x)| d\rhoe_t(y)\,dx\\
    &\le\|D^2\varphi\|_\infty\int_{\R^d}\int_{\Rd}V_\varepsilon(x-y)|y-x|\,d\rhoe_t(y)\,dx\\
    &=\varepsilon\|D^2\varphi\|_\infty\int_{\Rd}|z|V_1(z)\,dz,
\end{align*}
by means of the change of variable $z=\frac{x-y}{\varepsilon}$. The assertion follows by taking the supremum over $t\in[0,T]$.
\end{proof}
\begin{lem}
There exists a constant $C$ only depending on $\varphi$ and $V_1$ such that for all $\varepsilon > 0$ and a $\delta>0$ small enough
$$\|z^\varepsilon\|_{L^2([0,T];L^{2+\delta}(\Rd))}\le C.$$
\end{lem}
\begin{proof}
For almost every $x\in\Rd$ and $t\in[0,T]$, for $i=1,\ldots,d$, the non-negativity of $V_1$ and $\rhoe_t$ gives
$$   \left\vert \int_{\Rd} V_\varepsilon(x-y) \partial_{x_i} \varphi(y) d\rhoe_t(y) \right\vert 
\leq \int_{\Rd} V_\varepsilon(x-y) \vert \partial_{x_i} \varphi(y)  \vert d\rhoe_t(y) \leq 
\Vert \partial_{x_i} \varphi \Vert_{\infty} ~v^\varepsilon_t(x).$$ 
In particular, this implies
$$
|z_t^\varepsilon(x)|\le2\|\nabla\varphi\|_\infty|v_t^\varepsilon(x)|,
$$
thus $\Vert z^\varepsilon_t   \Vert_{L^p(\Rd)} \leq 2 \Vert \nabla \varphi\Vert_{\infty} \Vert v^\varepsilon_t \Vert_{L^p}$
for almost every $t\in[0,T]$, and $p\ge1$.

Proposition \ref{prop:en-ineq-mom-bound} and Lemma \ref{lem:h1-bound} entail existence of a constant $c$, independent of $\varepsilon$, such that
\begin{align*} 
\sup_{t \in [0,T]} \Vert v^\varepsilon_t \Vert_{L^2(\Rd)}\leq c \quad \mbox{and} \quad \int_0^T \Vert \nabla v_t^\varepsilon\Vert_{L^2}^2~dt  \leq c.
\end{align*}
Sobolev embedding theorems provide the estimate since
\[
\int_0^T\|v_t^\varepsilon\|_{L^{2+\delta}(\Rd)}^2dt\le c,
\]
for some constant still denoted by c.
\end{proof}

\begin{cor}\label{cor:convergence-z-0} 
The excess term converges to zero in $L^2([0,T]\times\Rd)$ as $\varepsilon\to0^+$. 
\end{cor}
\begin{proof}
The proof is a simple consequence of the interpolation inequality for $L^p$ functions and the previous Lemmas. More precisely, for $\alpha=\delta/(2(1+\delta))$ it holds
\begin{equation}
\begin{split}
\int_0^T\|z_t^\varepsilon\|_{L^2(\Rd)}^2dt&\le\int_0^T\|z_t^\varepsilon\|_{L^1(\Rd)}^{2\alpha}\|z_t^\varepsilon\|_{L^{2+\delta}(\Rd)}^{2(1-\alpha)}dt\\
&\le\left(\int_0^T\|z_t^\varepsilon\|_{L^1(\Rd)}^2dt\right)^\alpha\left(\int_0^T\|z_t^\varepsilon\|_{L^{2+\delta}(\Rd)}^2dt\right)^{1-\alpha}\le\varepsilon^{2\alpha} T^{\alpha}C(V_1,\varphi),
\end{split}
\end{equation}
which gives the result by letting $\varepsilon$ tend to $0$.
\end{proof}
\subsection{Convergence to the quadratic porous medium equation}
Let us consider the subsequence from Proposition \ref{prop:limit-rho}, still denoted by $\{\rhoe\}_\varepsilon$, which narrowly converges to the curve $\tilde{\rho}$. In the next Lemma we show that the corresponding smoothed subsequence, still denoted by $\{v^\varepsilon\}_{\varepsilon}$, converges to $\tilde{\rho}$ in the sense of distributions.
\begin{lem}\label{lem:limit-dist}
For any $t\in[0,T]$ and any $\varphi\in C_c^1(\Rd)$ we have
$$
\lim_{\varepsilon\to0^+}\int_\Rd \varphi(x)v_t^\varepsilon(x)\,dx=\int_{\Rd}\varphi(x)\,d\tilde{\rho}(t).
$$
\end{lem}
\begin{proof}
For any $t\in[0,T]$ and any $\varphi\in C_c^1(\Rd)$, by using the Definition of $v_t^\varepsilon$ we obtain:
\begin{align*}
\left|\int_\Rd\varphi(x)v_t^\varepsilon(x)\,dx-\int_\Rd\varphi(x)\,d\rhoe_t(x)\right|&=\left|\int_\Rd\varphi(x)(V_\varepsilon*\rhoe_t)(x)\,dx-\int_\Rd\varphi(x)\,d\rhoe_t(x)\right|\\
&=\left|\int_\Rd(\varphi*V_\varepsilon)(x)\,d\rhoe_t(x)-\int_\Rd\varphi(x)\,d\rhoe_t(x)\right|\\
&=\left|\int_\Rd[(\varphi*V_\varepsilon)(x)-\varphi(x)]\,d\rhoe_t(x)\right|\\
&\le\int_\Rd\int_\Rd|\varphi(x-y)-\varphi(x)|V_\varepsilon(y)\,dy\,d\rhoe_t(x)\\
&\le\|\varphi\|_\infty\int_\Rd |y|V_\varepsilon(y)\,dy\\
&=\varepsilon\|\varphi\|_\infty\int_\Rd |x|V_1(x)\,dx,
\end{align*}
which converges to $0$ as $\varepsilon\to0^+$ since $\int_\Rd|x|V_1(x)\,dx<+\infty$.
\end{proof}
We have now all the information to prove our first main result.
\begin{thm}\label{thm:existence-pme2}
Let  $\rho_0\in\mpdtard\cap L^2(\Rd)$ such that $\mh[\rho_0]<\infty$. The sequence $\{\rhoe\}_{0<\varepsilon\le1}$ of solutions to \eqref{eq:nlie} narrowly converges to the unique weak solution $\tilde{\rho}$ of \eqref{eq:pme2}. as $\varepsilon \rightarrow 0$ 
\end{thm}
\begin{proof}
Since $\rhoe$ is a weak solution to \eqref{eq:nlie}, for any $\varphi\in C^1_c(\Rd)$ and $t\in[0,T]$ it satisfies
\begin{align*}
    \int_\Rd\varphi(x)\,d\rho_t^\varepsilon(x)-\int_\Rd\varphi(x)\,d\rho_0(x)&=-\int_0^t \int_{\Rd} v^\varepsilon_t(x) \nabla \varphi(x) \cdot \nabla v^\varepsilon_t(x) \,dx\,dt\\
&\quad -\int_0^t \int_{\Rd} z^\varepsilon_t(x) \cdot \nabla v^\varepsilon_t(x) \,dx\,dt,
\end{align*}
as explained in \eqref{eq:weak-form-conv} and \eqref{eq:weak-form-z}. In view of Proposition \ref{prop:limit-rho}, Lemmas \ref{lem:h1-bound}, \ref{lem:limit-dist}, and Proposition \ref{prop:strong-convergence-v}, we know there exists a subsequence of $\rhoe(t)$ narrowly converging to $\tilde{\rho}\in L^2([0,T];H^1(\Rd))$, and, in particular, $\{v^\varepsilon\}_\varepsilon$ admits a subsequence such that
\begin{align*}
    &v^{\varepsilon_k}\to\tilde{\rho} \qquad\quad \mbox{ in } L^2([0,T];L^2(\Rd));\\
    \nabla& v^{\varepsilon_k}\rightharpoonup \nabla\tilde{\rho} \qquad \, \mbox{ in } L^2([0,T];L^2(\Rd)).
\end{align*}
Before letting $\varepsilon\to0^+$ and obtaining the result we need to further regularise the test function, $\varphi$, since Corollary \ref{cor:convergence-z-0} holds for test functions in $C^2_c(\Rd)$. In this regard, we consider a standard mollifier $\eta\in C_c^\infty(\Rd)$ and the corresponding sequence $\varphi^\sigma:=\eta^\sigma*\varphi\in C_c^\infty(\Rd)$, being $\eta^\sigma(x)=\sigma^{-d}\eta(x/\sigma^d)$ for any $x\in\Rd$ and $\sigma>0$. As consequence of the observations above and corollary \ref{cor:convergence-z-0}, by letting $\varepsilon\to0^+$ we obtain, for any $\sigma>0$ and $t\in[0,T]$,
\begin{align*}
     \int_\Rd\varphi^\sigma(x)\tilde\rho(t,x)\,dx= \int_\Rd\varphi^\sigma(x)\rho_0(x)\,dx-\int_0^t\int_\Rd\tilde\rho(s,x) \nabla \varphi^\sigma(x)\cdot \nabla \tilde\rho(s,x)\,dx\,ds.
    \end{align*}
Since $\varphi^\sigma$ converges uniformly to $\varphi$ on compact sets, we can let $\sigma\to0$ and obtain that $\tilde\rho$ is a weak solution to \eqref{eq:pme2} in the sense of Definition \ref{def:sol-pme2}. Uniqueness of weak solutions of \eqref{eq:pme2} is a known result, \textit{cf.~e.g.~}\cite{DalKen84,Vaz07}. By a standard contradiction argument we can prove the whole sequence $\rhoe$ narrowly converges to $\tilde{\rho}$, by further using uniqueness of weak solutions.
\end{proof}

\section{Extension to cross-diffusion systems}\label{sec:systems}

A key motivation for our approach is that it does not rely on geodesic convexity of the associated energy, hence it can be extended to scenarios when convexity fails. This is indeed the case of cross-diffusion systems, as shown in \cite{beck_matthes_zizza22}.
In what follows we extend the previous analysis to cross-diffusion systems, taking into account related works in literature, such as \cite{matthes_zinsl, DiFranEspFag}. More precisely, we derive the class of systems 
\begin{equation}\label{eq:cross-diff-sys}
\begin{cases}
	                 \partial_{t}\rhoo=\mbox{div}\left(\rhoo A_{1,1}\nabla \rhoo+\rhoo A_{1,2}\nabla \rhotwo\right), \\  \partial_{t}\rhotwo=\mbox{div}\left(\rhotwo A_{2,2}\nabla \rhotwo +\rhotwo A_{2,1}\nabla \rhoo\right),
\end{cases} \tag{CDS}
\end{equation}
which can be rewritten in matrix form as
\begin{equation}\label{eq:matrix-form-CDS}
\begin{pmatrix}
\partial_t \rhoo\\
\partial_t \rhotwo
\end{pmatrix}
= \mbox{div}
\left[\begin{pmatrix}
\rhoo A_{1,1} & \rhoo A_{1,2} \\
\rhotwo A_{2,1} & \rhotwo A_{2,2}
\end{pmatrix}
\begin{pmatrix}
\nabla \rhoo\\
\nabla \rhotwo
\end{pmatrix}\right].\tag{CDS-M}
\end{equation}
Hereafter, $\rhoo$ and $\rhotwo$ are two probability density accounting for \emph{two population species}; we consider two species for simplicity, though the results hold true for $M$ species, $M\in\mathbb{N}$. The coefficients $A_{i,j}$, for $i,j=1,2$, are so that
\begin{equation}\label{ass:a-diffusion}
\min\{A_{1,1},A_{2,2}\}>\frac{A_{1,2}+A_{2,1}}{2}\ge0.\tag{\textbf{A}}
\end{equation}
These coefficients can be, e.g., constant second order derivatives of a function $A$ depending on both species, i.e. $A=A(\rhoo,\rhotwo)$. A prototype function $A(\rhoo,\rhotwo)$ we can consider is $A(\rhoo,\rhotwo)=(c_1\rhoo+c_2\rhotwo)^2+c_3\rho_1^2$, for some constants $c_i>0$, $i=1,2,3$. Note that cross-diffusion is present in the case $A_{1,2}\neq0$ and $A_{2,1}\neq0$. In \cite{DiFranEspFag} the authors consider a general class of cross-diffusion systems, with the addition of nonlocal interaction terms, where $A(\rhoo,\rhotwo)$ is a nonlinear function modelling degenerate diffusion, for example of the form $A(\rhoo,\rhotwo)=\rhoo^{m_1}+\rhotwo^{m_2}+p(\rhoo+\rhotwo)$, being $m_1,\, m_2>1$ and $p$ regular enough. \cite{DiFranEspFag} provides existence of weak solutions by exploiting a semi-implicit version of the JKO scheme \cite{JKO98} in Wasserstein spaces.

Let us consider the functions $V_i:\Rd\to\R$ and  $U_{ij}:\Rd\to\R$ (or even a measure) for $i,j=1,2$ and $i\neq j$ such that $V_i$ satisfy \ref{ass:v1} and $V_i\in W^{1,2}(\Rd)$, while
\begin{enumerate}[label=\textbf{(U)}]
    \item \label{ass:u} $U_{ij}\in \mathcal{P}_1(\Rd)$ and it is even.  
\end{enumerate}

Let us define $H_i:=V_i*V_i$ and $K_i:=V_i*U_{ij}*V_j$ for $i,j=1,2$ and $i\neq j$. For any $\varepsilon>0$, consider the scaling $H_i^\varepsilon(x)=\varepsilon^{-d}H_i(\frac{x}{\varepsilon})$ and $K_i^\varepsilon(x)=\varepsilon^{-d}K_i(\frac{x}{\varepsilon})$, whence $H_i^\varepsilon=V_i^\varepsilon*V_i^\varepsilon$ and $K_i^\varepsilon=V_i^\varepsilon*U_{ij}^\varepsilon*V_j^\varepsilon$. Our goal is to show that, as $\varepsilon\to0^+$, weak-measure solutions $(\rhooe,\rhotwoe)$ to
\begin{equation}\label{eq:nl-int-sys}
\begin{cases}
\partial_{t}\rhoo=\mbox{div}\left(\rhoo A_{1,1}\nabla H_{1}^\varepsilon\ast\rhoo+\rhoo A_{1,2}\nabla K_{1}^\varepsilon\ast\rhotwo\right) \\

\partial_{t}\rhotwo=\mbox{div}\left(\rhotwo A_{2,2}\nabla H_{2}^\varepsilon\ast\rhotwo+\rhotwo A_{2,1}\nabla K_{2}^\varepsilon\ast\rhoo\right).
\end{cases} \tag{NLIS}
\end{equation}
converge to weak solutions to system \eqref{eq:cross-diff-sys}. In order to achieve such a goal we will use the smoother version of $(\rhooe,\rhotwoe)$ given by $(\voe,\vte)=(V_1^\varepsilon*\rhooe,V_2^\varepsilon*\rhotwoe)$. In \eqref{eq:nl-int-sys}, the kernels $H_i$ are the so-called self-interaction potentials since they model intra-specific interaction (among same species), whereas $K_i$ are known as cross-interaction potentials as they take into account inter-specific interaction (among different species).

\begin{rem}
The kernels $U_{ij}$ have been introduced to show a possible generalisation of this method, as this does not add other technical difficulties. Indeed, 
our assumptions allow to use $U_{ij}$ as a Dirac delta --- this is equivalent to excluding $U_{ij}$. However, we prefer to include a further regularisation as in other application it may be worth to introduce it in order to gain more regularity if $U_{ij}$ is absolutely continuous with respect to the Lebesgue measure.
\end{rem}

Below we state the Definitions of solutions used in this section.

\begin{defn}[Weak solution to \eqref{eq:cross-diff-sys}]\label{def:sol-cross-diff-sys}
A weak solution to the cross-diffusion system \eqref{eq:cross-diff-sys} on the time interval $[0,T]$ with initial datum $\rrho^0\in (\mptra\cap L^2(\Rd))^2$ is a curve $\rrho\in C([0,T];\mptrd\times\mptrd)$ satisfying the following properties:
\begin{enumerate}
    \item for almost every $t\in[0,T]$ the measure $\rrho(t)$ has a density with respect to the Lebesgue measure, still denoted by $\rrho(t)$, and $\rrho\in L^2([0,T];H^1(\Rd))\times L^2([0,T];H^1(\Rd))$;
    \item for any $\varphi,\phi\in C^1_c(\R^d)$ and all $t\in[0,T]$ it holds
\begin{subequations}
\begin{align*}
        \int_\Rd\varphi(x)\rhoo(t,x)\,dx&= \int_\Rd\varphi(x)\rhoo^0(x)\,dx-A_{1,1}\int_0^t\int_\Rd\rhoo(s,x) \nabla \varphi(x)\cdot \nabla \rhoo(s,x)\,dx\,ds\\
        &\quad-A_{1,2}\int_0^t\int_\Rd\rhoo(s,x)\nabla\varphi(x)\cdot\nabla\rho_2(s,x)\,dx\,ds,
\end{align*}
\begin{align*}
        \int_\Rd\phi(x)\rhotwo(t,x)\,dx&= \int_\Rd\phi(x)\rhotwo^0(x)\,dx-A_{2,2}\int_0^t\int_\Rd\rhotwo(s,x) \nabla \phi(x)\cdot \nabla \rhotwo(s,x)\,dx\,ds\\
        &\quad-A_{2,1}\int_0^t\int_\Rd\rhotwo(s,x)\nabla\phi(x)\cdot\nabla\rhoo(s,x)\,dx\,ds.
    \end{align*}
\end{subequations}
\end{enumerate}
\end{defn}

\begin{defn}[Weak measure solution to \eqref{eq:nl-int-sys}]\label{def:weak-meas-sol-system}
A narrowly continuous curve $\rrho^\varepsilon:[0,T]\to\mptrd\times\mptrd$, mapping $t\in[0,T]\mapsto\rrho_t^\varepsilon\in\mptrd\times\mptrd$, is a weak measure solution to \eqref{eq:nl-int-sys} if, for every $\varphi,\phi\in C^1_c(\Rd)$ and any $t\in[0,T]$, it holds
\begin{subequations}\label{subeq:weak-form}
\begin{equation}\label{eq:weak-form-1}
    \begin{split}
    \int_\Rd\varphi(x)d(\rho_{1,t}^\varepsilon-\rhoo^0)(x)&=-\frac{A_{1,1}}{2}\int_0^t\!\!\iint_\Rdd(\nabla\varphi(x)-\nabla\varphi (y))\cdot\nabla H_1^\varepsilon(x-y)d\rho_{1,r}^\varepsilon(y)d\rho_{1,r}^\varepsilon(x)dr\\
    &\quad-A_{1,2}\int_0^t\!\!\iint_{\Rdd}\nabla\varphi(x)\cdot \nabla K_1^\varepsilon(x-y)d\rho_{2,r}^\varepsilon(y)d\rho_{1,r}^\varepsilon(x)dr
    \end{split}
\end{equation}
\begin{equation}\label{eq:weak-form-2}
    \begin{split}
    \int_\Rd\phi(x)d(\rho_{2,t}^\varepsilon-\rhotwo^0)(x)&=-\frac{A_{2,2}}{2}\int_0^t\!\!\iint_\Rdd(\nabla\phi(x)-\nabla\phi(y))\cdot\nabla H_2^\varepsilon(x-y)d\rho_{2,r}^\varepsilon(y)d\rho_{2,r}^\varepsilon(x)dr\\
    &\quad-A_{2,1}\int_0^t\!\!\iint_{\Rdd}\nabla\phi(x)\cdot \nabla K_2^\varepsilon(x-y)d\rho_{1,r}^\varepsilon(y)d\rho_{2,r}^\varepsilon(x)dr
    \end{split}
    \end{equation}
    \end{subequations}
\end{defn}

\begin{rem}\label{rem:prod-space}
We denote elements of a product space by using bold symbols, e.g.
\[
\bm{\rho}=(\rhoo,\rhotwo)\in\mptrd\times\mptrd, \quad \mbox{or} \quad \bm{x}=(x_1,x_2)\in\Rd\times\Rd.
\]
The Wasserstein distance of order two in the product space is defined as follows
\[
\mw_2^2(\bm{\mu},\bm{\nu})=d_W^2(\mu_1,\nu_1)+d_W^2(\mu_2,\nu_2)
\]
for all $\bm{\mu},\bm{\nu}\in\mptrd\times\mptrd$.
\end{rem}

\subsection{Nonlocal Interaction System}\label{sec:nonloc-int-sys}
Following the reference paper for nonlocal interaction systems \cite{DFF}, we apply a \textit{semi-implicit} version of the JKO scheme, \cite{JKO98}, in order to obtain a priori estimates on solutions to \eqref{eq:nl-int-sys} and their smoothed version $v_i^\varepsilon=V_i^\varepsilon*\rho_i^\varepsilon$, for $i=1,2$. The semi-implict JKO scheme allows to prove existence of solutions to a class of systems of nonlocal interaction PDEs that do not exhibit a gradient flow structure, which is indeed the case when the cross-interaction potentials are not proportional, i.e. $K_1\neq \alpha K_2$ for a positve $\alpha$, cf. \cite{DFF} for further details.

Let $\varepsilon>0$ be fixed and finite. We consider an initial datum $\bm{\rho}^0\in(\mptrd\cap L^2(\Rd))^2$ and we introduce the \textit{relative energy} functional $\mf_\varepsilon:\mptrd\times\mptrd\to (-\infty,+\infty]$ defined as follows: let $\bm{\nu}\in\mptrd^2$ be a fixed (time independent) measure, for all $\bm{\mu}\in\mptrd^2$ we set
\begin{equation}\label{eq:relative-en-funct}
\begin{split}
\mf_\varepsilon[\bm{\mu}|\bm{\nu}]&:= \frac{A_{1,1}}{2}\int_{\Rd}H_1^\varepsilon \ast \mu_1\,d\mu_1 +A_{1,2} \int_{\Rd}K_1^\varepsilon \ast \nu_2\,d\mu_1\\
&\quad +\frac{A_{2,2}}{2}\int_{\Rd}H_2^\varepsilon * \mu_2\,d\mu_2 + A_{2,1}\int_{\Rd}K_2^\varepsilon * \nu_1\,d\mu_2.
\end{split}
\end{equation}
The above functional is referred to as \textit{relative energy} since it accounts for the energy at the state $\mmu$ given the state $\nnu$, which only affects the cross-interaction part of the functional. The latter observation suggests to rewrite the functional $\mf_\varepsilon$ as sum of two contributions, i.e. ``self'' and ``cross'' interactions. This will not only simplify notations, but also single out the two parts in the semi-implicit JKO scheme. Let us set
$$
\mh_\varepsilon[\mmu]:= \frac{A_{1,1}}{2}\int_{\Rd}H_1^\varepsilon * \mu_1\,d\mu_1 + \frac{A_{2,2}}{2}\int_{\Rd}H_2^\varepsilon * \mu_2\,d\mu_2,
$$
and
$$
\mk_\varepsilon[\mmu|\nnu]:=A_{1,2}\int_{\Rd}K_1^\varepsilon * \nu_2\,d\mu_1 + A_{2,1}\int_{\Rd}K_2^\varepsilon * \nu_1\,d\mu_2.
$$
Then, we can rewrite $\mf_\varepsilon$ as the sum of the previous functionals, i.e.
$$
\mf_\varepsilon[\mmu|\nnu]=\mh_\varepsilon[\mmu]+\mk_\varepsilon[\mmu|\nnu].
$$
Note that the part $\mh_\varepsilon[\mmu]$ is treated implicitly in the JKO scheme as usual, whereas $\mk[\mmu|\nnu]$ contains terms that are treated explicitly.
\begin{rem}\label{rem:uniform-bound-relative-energy}
Note that for $\bm{\rho}^0\in(\mptrd\cap L^2(\Rd))^2$ the relative energy at the initial datum is finite. More precisely, for any $i,j=1,2$ we have
\begin{equation}
    \begin{split}
    \int_{\Rd}(H_i^\varepsilon*\rhoi^0)(x)\rhoi^0(x)\,dx&=\int_\Rd |(V_i^\varepsilon*\rhoi^0)(x)|^2\,dx=\|V_i^\varepsilon*\rhoi^0\|_{L^2(\Rd)}^2\\
    &\le\|V_i^\varepsilon\|_{L^1(\Rd)}^2\|\rhoi^0\|_{L^2(\Rd)}^2=\|V_i\|_{L^1(\Rd)}^2\|\rhoi^0\|_{L^2(\Rd)}^2<\infty,
    \end{split}
\end{equation}
which implies
\begin{equation}\label{eq:bound-self-initial-system}
    \begin{split}
    \mh_\varepsilon[\rrho^0]&\le \frac{A_{1,1}}{2}\|V_1\|_{L^1(\Rd)}^2\|\rhoo^0\|_{L^2(\Rd)}^2+\frac{A_{2,2}}{2}\|V_2\|_{L^1(\Rd)}^2\|\rhotwo^0\|_{L^2(\Rd)}^2\\
    &\le\max\left\{\frac{A_{1,1}}{2}\|V_1\|_{L^1(\Rd)}^2,\frac{A_{2,2}}{2}\|V_2\|_{L^1(\Rd)}^2\right\}\|\rrho^0\|_{L^2}^2,
    \end{split}
\end{equation}
and
\begin{equation}\label{eq:bound-cross-initial}
    \begin{split}
    \int_{\Rd}(K_i^\varepsilon\ast\nu_j)(x)\rhoi^0(x)\,dx&=\int_{\Rd}(V_i^\varepsilon\ast\rhoi^0)(x)(U_{ij}^\varepsilon\ast(V_j^\varepsilon\ast\nu_j))(x)\,dx\\
    &\le\|V_i^\varepsilon\ast\rhoi^0\|_{L^2(\Rd)}\|U_{ij}^\varepsilon\ast(V_j^\varepsilon\ast\nu_j)\|_{L^2(\Rd)}\\
    &\le\|V_i^\varepsilon\|_{L^1(\Rd)}\|\rhoi^0\|_{L^2(\Rd)} \|V_j^\varepsilon\|_{L^2(\Rd)}\nu_j(\Rd)\\
    &=\frac{1}{\varepsilon^{d/2}}\|V_i\|_{L^1(\Rd)}\|\rhoi^0\|_{L^2(\Rd)} \|V_j\|_{L^2(\Rd)}<+\infty.
    \end{split}
\end{equation}
\end{rem}
The sequence is defined via the semi-implicit JKO scheme:
\begin{itemize}
    \item fix a time step $\tau>0$ such that $\rrhotau^0:=\rrho^0$;
    \item for a given $\rrhotaune\in(\mptrd)^2$, choose
\begin{equation}\label{eq:semijko}
    \rrhotaunne\in\argmin_{\rrho\in(\mptrd)^2}\left\{\frac{\mw_2^2(\rrhotaune,\rrho)}{2\tau}+\mf_\varepsilon[\rrho|\rrhotaune]\right\}.
\end{equation}
\end{itemize}
Let $T>0$, $N:=\left[\frac{T}{\tau}\right]$, and consider the piecewise constant interpolation
$$
\rrhotaue(t)=\rrhotaune \qquad t\in((n-1)\tau,n\tau],
$$
being $\rrhotaune=(\rhootaune,\rhottaune)$ defined in \eqref{eq:semijko}.

\begin{prop}[Narrow compactness -- energy $\&$ moments bound]\label{prop:en-ineq-mom-bound-system}
There exists an absolutely continuous curve $\tilde{\rrho}^\varepsilon: [0,T]\rightarrow\mptrd\times\mptrd$ such that the piecewise constant interpolation $\rrhotaue$ admits a subsequence $\rrho_{\tau_k}^\varepsilon$ narrowly converging to $\tilde{\rrho}^\varepsilon$ uniformly in $t\in[0,T]$ as $k\rightarrow +\infty$. Moreover, for any $t\in[0,T]$, the following uniform bounds in $\tau$ and $\varepsilon$ hold
\begin{subequations}
\begin{align}
    \mh_\varepsilon[\tilde{\rrho}^\varepsilon(t)]&\le c \|\rrho^0\|_{L^2}^2,\\
    m_2(\tilde{\rrho}^\varepsilon(t))&\le2m_2(\rrho^0)+\tilde{c},
\end{align}
\end{subequations}
where $c=\max\left\{\frac{A_{1,1}}{2}\|V_1\|_{L^1(\Rd)}^2,\frac{A_{2,2}}{2}\|V_2\|_{L^1(\Rd)}^2\right\}$ and $\tilde{c}=4cT\|\rrho^0\|_{L^2}^2$.
\end{prop}
\begin{proof}
From the Definition of the sequence $\{\rrhotaune\}_{n\in\mathbb{N}}$ it holds
\begin{equation}\label{eq:scheme1}
\begin{split}
&\frac{1}{2\tau}\mw_2^2(\rrhotaune,\rrhotaunne)\le\mf_\varepsilon[\rrhotaune|\rrhotaune]-\mf_\varepsilon[\rrhotaunne|\rrhotaune]\\
&=\sum_{i=1}^{2}\frac{A_{i,i}}{2}\left(\int_{\Rd}H_i^\varepsilon * \rhoitaune\,d\rhoitaune-\int_{\Rd}H_i^\varepsilon * \rhoitaunne\,d\rhoitaunne\right)
\\
&\quad+\sum_{i\ne j}A_{i,j}\left(\int_{\Rd}K_i^\varepsilon * \rhojtaune\,d\rhoitaune-\int_{\Rd}K_i^\varepsilon * \rhojtaune\,d\rhoitaunne\right)\\
&=\mh_\varepsilon[\rrhotaune]-\mh_\varepsilon[\rrhotaunne]+\sum_{i\ne j}A_{i,j}\left(\int_{\Rd}K_i^\varepsilon * \rhojtaune\,d\rhoitaune-\int_{\Rd}K_i^\varepsilon * \rhojtaune\,d\rhoitaunne\right).
\end{split}
\end{equation}
In order to have an estimate on the energy, we make use of the $L^\infty$ bound $\|\nabla V_i^\varepsilon*V_j^\varepsilon\|_{L^\infty(\Rd)}\le \varepsilon^{-(d+1)}\|\nabla V_i\|_{L^2(\Rd)}\|V_j\|_{L^2(\Rd)}$ as follows for $i,j=1,2$, $i\neq j$. First note that
\begin{align*}
 & \left|\int_{\Rd}K_i^\varepsilon * \rhojtaune\,d\rhoitaune-\int_{\Rd}K_i^\varepsilon * \rhojtaune\,d\rhoitaunne \right|\\
 & \ = \left|\iint_{\Rd\times\Rd}K_i^\varepsilon(x-y)\,d\rhojtaune(y)\,d\rhoitaune(x)-\iint_{\Rd\times\Rd} K_i^\varepsilon(t-y)\,d\rhojtaune(y)\,d\rhoitaunne(t)\right|\\
 & \ = \left|\iiint_{\R^d\times\R^d\times\R^d}(K_i^\varepsilon(x-y)-K_i^\varepsilon(t-y))\,d \gamma_{i,\tau}^{\varepsilon,n}(x,t)\,d \rhojtaune(y)\right|
\end{align*}
where $\gamma^{\varepsilon,n}_{i,\tau}\in \Gamma_o(\rhoitaune,\rhoitaunne)$ is an optimal transport plan connecting $\rhoitaune$ to $\rhoitaunne$. Now, due to the $L^\infty$ bound on $\nabla V_i^\varepsilon*V_j^\varepsilon$ we get
\begin{align*}
 &  \left|\iiint_{\R^d\times\R^d\times\R^d}(K_i^\varepsilon(x-y)-K_i^\varepsilon(t-y))\,d \gamma_{i,\tau}^{\varepsilon,n}(x,t)\,d \rhojtaune(y)\right|\\
 &= \left |\iiint_{\Rd\times\Rd\times\Rd}\left\{\int_{\Rd}[(V_i^\varepsilon*V_j^\varepsilon)(x-y-z)-(V_i^\varepsilon*V_j^\varepsilon)(t-y-z)] U_{ij}^\varepsilon(z) \right\}\,d \gamma_{i,\tau}^{\varepsilon,n}(x,t)\,d \rhojtaune(y) \right |\\
 &\le \|\nabla V_i^\varepsilon*V_j^\varepsilon\|_{L^\infty(\Rd)}  \iint_{\Rd\times\Rd}|x-t|\,d \gamma_{i,\tau}^{\varepsilon,n}(x,t)\\
 & \ \leq  \frac{1}{\varepsilon^{d+1}}\|\nabla V_i\|_{L^2(\Rd)}\|V_j\|_{L^2(\Rd)} d_W(\rhoitaune,\rhoitaunne) \leq \frac{1}{4\tau}d_W^2(\rhoitaune,\rhoitaunne) + C\frac{\tau}{\varepsilon^{2(d+1)}},
\end{align*}
where $C$ is a positive constant independent of $\tau$. By using the latter estimate in \eqref{eq:scheme1} we obtain
\begin{align}\label{eq:basic-ineq-system}
    \frac{1}{4\tau}\mw_2^2(\rrhotaune,\rrhotaunne)\le\mh_\varepsilon[\rrhotaune]-\mh_\varepsilon[\rrhotaunne]+C\frac{\tau}{\varepsilon^{2(d+1)}},
\end{align}
which implies $\mh_\varepsilon[\rrhotaunne]\le\mh_\varepsilon[\rrhotaune] +C\frac{\tau}{\varepsilon^{2(d+1)}}$, and, in particular, the following bound for the self-interaction part, $\mh_\varepsilon$, of the relative energy $\mf_\varepsilon$:
\begin{equation}\label{eq:self-int-energy-bound}
    \mh_\varepsilon[\rrhotaune]\le\mh_\varepsilon[\rrho^0]+C\frac{n\tau}{\varepsilon^{2(d+1)}}\le\mh_\varepsilon[\rrho^0]+C\frac{T}{\varepsilon^{2(d+1)}}, \qquad \forall n\in\mathbb{N}.
\end{equation}
By summing up over $k$ inequality \eqref{eq:basic-ineq-system}, we obtain
\begin{align}
\sum_{k=m}^n\frac{\mw_2^2(\rrho_{\tau}^{\varepsilon,k},\rrho_{\tau}^{\varepsilon, k+1})}{4\tau}\le \mh_\varepsilon[\rrho_{\tau}^{\varepsilon, m}]-\mh_\varepsilon[\rrhotaunne]+C\frac{\tau}{\varepsilon^{2(d+1)}}(n-m+1).
\end{align}

The non-negativity of $\mh_\varepsilon$ and the energy inequality \eqref{eq:self-int-energy-bound} allow us to improve the above inequality as follows
\begin{align}\label{eq:total-square-en-estim-system}
    \sum_{k=m}^n\frac{\mw_2^2(\rrho_{\tau}^{\varepsilon,k},\rrho_{\tau}^{\varepsilon, k+1})}{4\tau}\le\mh_\varepsilon[\rrho^0]+C\frac{T}{\varepsilon^{2(d+1)}} +C\frac{\tau}{\varepsilon^{2(d+1)}}(n-m+1).
\end{align}
In particular, by using the bound \eqref{eq:bound-self-initial-system} in Remark \ref{rem:uniform-bound-relative-energy}, the above inequality implies
\begin{align*}
\mw_2^2(\rrho^0,\rrhotaue(t))&\le 4T \mh_\varepsilon[\rrho^0]+C\frac{T^2}{\varepsilon^{2(d+1)}}\\
&\le 2T\max\left\{A_{1,1}\|V_1\|_{L^1(\Rd)}^2,A_{2,2}\|V_2\|_{L^1(\Rd)}^2\right\}\|\rrho^0\|_{L^2}^2+C\frac{T^2}{\varepsilon^{2(d+1)}},
\end{align*}
whence we obtain the second order moments are uniformly bounded in $\tau $ on $[0,T]$ in view of Remark \ref{rem:mom-ineq}, i.e. 
\begin{equation}\label{eq:mom-inequality-system}
\begin{split}
    m_2(\rrhotaue(t))&\le 2m_2(\rrho^0)+2\mw_2^2(\rrho^0,\rrhotaue(t))\\&\le 2m_2(\rrho^0)+2T\max\left\{A_{1,1}\|V_1\|_{L^1(\Rd)}^2,A_{2,2}\|V_2\|_{L^1(\Rd)}^2\right\}\|\rrho^0\|_{L^2}^2+C\frac{T^2}{\varepsilon^{2(d+1)}}.
\end{split}
\end{equation}
Now, let us consider $0\le s<t$ such that $s\in((m-1)\tau,m\tau]$ and $t\in((n-1)\tau,n\tau]$ (which implies $|n-m|<\frac{|t-s|}{\tau}+1$); by the Cauchy-Schwartz inequality, \eqref{eq:total-square-en-estim-system} and again the bound \eqref{eq:bound-self-initial-system} in Remark \ref{rem:uniform-bound-relative-energy}, we obtain
\begin{equation}\label{eq:holder-cont-system}
\begin{split}
\mw_2(\rrhotaue(s),\rrhotaue(t))&\le\sum_{k=m}^{n-1}\mw_2(\rrho_{\tau}^{\varepsilon, k},\rrho_{\tau}^{\varepsilon,k+1})\le\left(\sum_{k=m}^{n-1}\mw_2^2(\rrho_{\tau}^{\varepsilon,k},\rrho_{\tau}^{\varepsilon, k+1})\right)^{\frac{1}{2}}|n-m|^{\frac{1}{2}}\\&\le c\left(\sqrt{1+\frac{T}{\varepsilon^{2(d+1)}}}\right) \left(\sqrt{|t-s|}+\sqrt{\tau}\right),
\end{split}
\end{equation}\noindent
where $c=c(A_{1,1},A_{2,2},\rrho^0,T)$ is a positive constant.
Thus $\rrhotaue$ is $\frac{1}{2}$-H\"{o}lder equi-continuous, up to a negligible error of order $\sqrt{\tau}$. By using a refined version of Ascoli-Arzel\`{a}'s theorem (see \cite{AGS}, Section 3), we obtain $\rrhotaue$ admits a subsequence narrowly converging to a limit $\tilde{\rrho}^\varepsilon$ as $\tau\to0^+$ uniformly on $[0,T]$. Since $|\cdot|^2$ and $H_i^\varepsilon$ are lower semi-continuous and bounded from below, a refined analysis gives for any $t\in[0,T]$
\begin{align*}
    \mh_\varepsilon[\tilde{\rrho}^\varepsilon]\le\liminf_{k\to+\infty}\mh_\varepsilon[\rrho_{\tau_k}^\varepsilon]\le \mh_\varepsilon[\rrho^0]\le\max\left\{\frac{A_{1,1}}{2}\|V_1\|_{L^1(\Rd)}^2,\frac{A_{2,2}}{2}\|V_2\|_{L^1(\Rd)}^2\right\}\|\rrho^0\|_{L^2}^2,
\end{align*}
and
\begin{align*}
\int_\Rd|x|^2\,d\tilde{\rho_i}^\varepsilon(t)(x)&\le\liminf_{k\to+\infty}\int_\Rd|x|^2\,d\rho_{\tau_{k},i}^\varepsilon(t)(x)\\
&\le2m_2(\rrho^0)+2T\max\left\{A_{1,1}\|V_1\|_{L^1(\Rd)}^2,A_{2,2}\|V_2\|_{L^1(\Rd)}^2\right\}\|\rrho^0\|_{L^2}^2
\end{align*}
whence the assertion follows by applying the above inequalities to \eqref{eq:self-int-energy-bound} and \eqref{eq:mom-inequality-system}.
\end{proof}
As direct consequence of the previous Lemma, we actually have narrow compactness in $\varepsilon$ by following the same argument used in Proposition \ref{prop:limit-rho}; hence we omit the proof.
\begin{cor}\label{cor:limit-eps-rho}
There exists an absolutely continuous curve $\tilde{\rrho}: [0,T]\rightarrow\mptrd\times\mptrd$ such that $\tilde{\rrho}^\varepsilon$ admits a subsequence $\tilde{\rrho}^{\varepsilon_k}$ narrowly converging to $\tilde{\rrho}$ uniformly in $t\in[0,T]$ as $k\rightarrow +\infty$.
\end{cor}

As in section \ref{sec:nlie}, the limiting curve $\tilde{\rrho}^\varepsilon$ obtained in Proposition \ref{prop:en-ineq-mom-bound-system} is a weak measure solution of the nonlocal interaction system \eqref{eq:nl-int-sys} in the sense of Definition \ref{def:weak-meas-sol-system}. The proof of this result is similar to \ref{thm:existence-nlie}, and more details can be found in \cite[Theorem 3.3]{DiFranEspFag}. For this reason we omit the proof. Let us stress that the cross-interaction potentials, $K_i^\varepsilon$ are $C^1(\Rd)$ since continuity at zero is obtained using Lebesgue dominated convergence theorem, exploiting that $U_{ij}\in \mP_1(\Rd)$ and $\nabla V_i^\varepsilon\in L^1(\Rd)$. This is needed as we cannot cope with a possible discontinuity of $\nabla K_i^\varepsilon$ at zero, as for the self-interaction kernels, cf.~\cite{DFF,DiFranEspFag}. 

\begin{thm}\label{thm:existence-nlie-system}
The curve $\tilde{\rrho}^\varepsilon$ is a weak measure solution to the system \eqref{eq:nl-int-sys} according to Definition \ref{def:weak-meas-sol-system}.
\end{thm}

\subsection{Compactness in $\varepsilon$}

As in the one-species case, stronger convergence can be obtained for the sequence $\{\bm{v}^\varepsilon\}_\varepsilon$, being $v_i^\varepsilon(t):=V_i^\varepsilon*\rhoi^\varepsilon(t)$ for any $t\in[0,T]$. We use the flow interchange technique by considering a decoupled system of heat equations as auxiliary flow, i.e.
\begin{equation}\label{eq:heat-sys}
\begin{cases}
\partial_{t}\eta_1=\Delta\eta_1 \\
\partial_{t}\eta_2=\Delta\eta_2,
\end{cases}
\end{equation}
and the entropy as auxiliary functional, that is
\begin{equation}\label{eq:aux-func-sys}
\mathrm{E}[\eta_1,\eta_2]=
\begin{cases}
\int_{\Rd}[\eta_1(x)\log\eta_1(x)+\eta_2(x)\log\eta_2(x)]\,dx, &\eta_1\log\eta_1,\eta_2\log\eta_2\in L^1(\Rd);\\
+\infty & \text{otherwise}.
\end{cases}
\end{equation}
For any $\nnu\in\mptrd$ such that $\mathrm{E}(\nnu)<+\infty$, we denote by $\bm{S}_{\mathrm{E}}^t\nnu:=(S_{\mathrm{E},1}^t\nu_1,S_{\mathrm{E},2}^t\nu_2)$ the solution at time $t$ to system \eqref{eq:heat-sys} coupled with an initial value $\nnu$ at $t=0$. Moreover, for every $\rrho\in(\mptrd)^2$ and a given $\mmu\in(\mptrd)^2$, the dissipation of $\mf_\varepsilon$ along $\bm{S}_\mee$ by 
$$
\bm{D}_\mee\mf_\varepsilon(\rrho|\mmu):=\limsup_{s\downarrow0}\left\{\frac{\mf_\varepsilon[\rrho|\mmu]-\mf_\varepsilon[\bm{S}_{\mathrm{E}}^s\rrho|\mmu]}{s}\right\}.
$$


Below we prove an $L^2_tH^1_x$ bound crucial for the application of the Rossi-Savaré version of the Aubin-Lions Lemma. 

\begin{lem}\label{lem:system-h1-bound}
Let $\rrho^0\in(\mpdtard\cap L^2(\Rd))^2$ such that $\mathrm{E}[\rrho^0]<\infty$. There exists a constant $C=C(\rrho^0,A_{1,1},A_{2,2},V_1,V_2,T)$ such that, for any $\varepsilon>0$,
\begin{align}
    \sum_{i=1}^2\|v_i^\varepsilon\|_{L^2([0,T];H^1(\Rd))}^2\le C.
\end{align}
\end{lem}
\begin{proof}
First of all we obtain a uniform $L^2$ bound in time and space from Proposition \ref{prop:en-ineq-mom-bound-system} by noticing that
\begin{equation}
    \begin{split}
    \sum_{i=1}^2A_{i,i}\|V_i^\varepsilon*\rho_{i,\tau}^\varepsilon\|_{L^2([0,T];L^2(\Rd))}^2&=\sum_{i=1}^2A_{i,i}\int_0^T\int_\Rd|[V_i^\varepsilon*\rho_{i,\tau}^\varepsilon(t)](x)|^2dx\\
    &=\sum_{i=1}^2A_{i,i}\int_0^T\int_\Rd [H_i^\varepsilon*\rho_{i,\tau}^\varepsilon(t)](x)d\rho_{i,\tau}^\varepsilon(t)(x)\\
    &=2\int_0^T\mh_\varepsilon[\rrhotaue(t)]dt\le C(\rrho^0,A_{1,1},A_{2,2},T).
    \end{split}
\end{equation}
Since $A_{i,i}\neq0$ for any $i$, we can divide and get the $L^2$ bound, arguing as in Proposition \ref{lem:h1-bound}. Let us now focus on the bound for the gradient. For all $s>0$, using the minimising property of $\rrhotaunne$ in \eqref{eq:semijko} we have
$$
\frac{1}{2\tau}\mw_2^2(\rrhotaunne,\rrhotaune)+\mf_\varepsilon[\rrhotaunne|\rrhotaune]\le\frac{1}{2\tau}\mw_2^2(\bm{S}_{\mee}^s\rrhotaunne,\rrhotaune)+\mf_\varepsilon[\bm{S}_{\mee}^s\rrhotaunne|\rrhotaune],
$$
whence, dividing by $s>0$ and passing to the $\limsup$ as $s\downarrow0$,
\begin{equation}\label{eq:controlf}
\tau\bm{D}_{\mee}\mf_\varepsilon[\rrhotaunne|\rrhotaune]\le\frac{1}{2}\frac{d^+}{dt}\bigg(\mw_2^2(\bm{S}_{\mee}^t\rrhotaunne,\rrhotaune)\bigg)\Big|_{t=0}\overset{\bm{(E.V.I.)}}{\le}\mee[\rrhotaune]-\mee[\rrhotaunne].
\end{equation}
In the last inequality we used that $\bm{S}_\mee$ is a $0$-flow. Focusing on the left hand side of \eqref{eq:controlf}, we note
\begin{equation}\label{eq:integral-form-dis-f}
\begin{split}
\bm{D}_{\mee}\mf_\varepsilon[\rrhotaunne|\rrhotaune]&=\limsup_{s\downarrow0}\left\{\frac{\mf_\varepsilon[\rrhotaunne|\rrhotaune]-\mf_\varepsilon[\bm{S}_{\mee}^s\rrhotaunne|\rrhotaune]}{s}\right\}\\&=\limsup_{s\downarrow0}\int_0^1\left(-\frac{d}{dz}\Big|_{z=st}\mf_\varepsilon[\bm{S}_{\mee}^{z}\rrhotaunne|\rrhotaune]\right)\,dt.
\end{split}
\end{equation}
Thus, we now compute the time derivative inside the above integral, by using integration by parts and keeping in mind the $C^\infty$ regularity of the solution to the heat equation:
\begin{equation*}\label{eq:deriv-dis-f}
\begin{split}
\frac{d}{dt}\mf_\varepsilon[\bm{S}_{\mee}^t\rrhotaunne|\rrhotaune]&=-A_{1,1}\int_{\Rd}\nabla (H_1^\varepsilon*S_{\mee,1}^t\rhootaunne)(x)\nabla S_{\mee,1}^t\rhootaunne(x)\,dx\\
&\quad-A_{2,2}\int_{\Rd}\nabla (H_2^\varepsilon*S_{\mee,2}^t\rhottaunne)(x)\nabla S_{\mee,2}^t\rhottaunne(x)\,dx\\
&\quad-A_{1,2}\int_{\Rd}\nabla (K_1^\varepsilon*\rhottaune)(x)\nabla S_{\mee,1}^t\rhootaunne(x)\,dx\\
&\quad-A_{2,1}\int_{\Rd}\nabla (K_2^\varepsilon*\rhootaune)(x)\nabla S_{\mee,2}^t\rhottaunne(x)\,dx\\
&=-A_{1,1}\int_{\Rd}|\nabla V_1^\varepsilon*S_{\mee,1}^t\rhootaunne(x)|^2\,dx\\
&\quad-A_{2,2}\int_{\Rd}|\nabla V_2^\varepsilon*S_{\mee,2}^t\rhottaunne(x)|^2\,dx\\
&\quad-A_{1,2}\int_{\Rd}\nabla (V_1^\varepsilon*S_{\mee,1}^t\rhootaunne)(x) [\nabla ( U_{12}^\varepsilon* V_2^\varepsilon*\rhottaune)](x)\,dx\\
&\quad-A_{2,1}\int_{\Rd}\nabla(V_2^\varepsilon* S_{\mee,2}^t\rhottaunne)(x)[\nabla (U_{21}^\varepsilon* V_1^\varepsilon*\rhootaune)](x)\,dx
\end{split}
\end{equation*}
\begin{equation*}
\begin{split}
&\le -A_{1,1}\int_{\Rd}|\nabla V_1^\varepsilon*S_{\mee,1}^t\rhootaunne(x)|^2\,dx\\
&\quad-A_{2,2}\int_{\Rd}|\nabla V_2^\varepsilon*S_{\mee,2}^t\rhottaunne(x)|^2\,dx\\
&\quad+A_{1,2}\|\nabla (V_1^\varepsilon*S_{\mee,1}^t\rhootaunne)\|_{L^2(\Rd)}\|\nabla (U_{12}^\varepsilon*  V_2^\varepsilon*\rhottaune)\|_{L^2(\Rd)}\\
&\quad+A_{2,1}\|\nabla(V_2^\varepsilon* S_{\mee,2}^t\rhottaunne)\|_{L^2(\Rd)}\|\nabla (U_{21}^\varepsilon* V_1^\varepsilon*\rhootaune)\|_{L^2(\Rd)}\\
&\le-\left(A_{1,1}-\frac{A_{1,2}}{2}\right)\int_{\Rd}|\nabla V_1^\varepsilon*S_{\mee,1}^t\rhootaunne(x)|^2\,dx\\
&\quad - \left(A_{2,2}-\frac{A_{2,1}}{2}\right)\int_{\Rd}|\nabla V_2^\varepsilon*S_{\mee,2}^t\rhottaunne(x)|^2\,dx\\
&\quad +\frac{A_{1,2}}{2}\int_{\Rd}|\nabla ( U_{12}^\varepsilon* V_2^\varepsilon*\rhottaune)(x)|^2\,dx\\
&\quad+\frac{A_{2,1}}{2}\int_{\Rd}|\nabla (U_{21}^\varepsilon*  V_1^\varepsilon*\rhootaune)(x)|^2\,dx.
\end{split}
\end{equation*}
The above inequality, together with \eqref{eq:controlf} and \eqref{eq:integral-form-dis-f}, implies
\begin{equation*}
\begin{split}
    \tau\liminf_{s\downarrow0}\!\!\int_0^1\!\!\!\int_{\Rd}\!\!&\left(A_{1,1}\!-\!\frac{A_{1,2}}{2}\right)|\nabla V_1^\varepsilon*S_{\mee,1}^{st}\rhootaunne(x)|^2\!+\!\left(A_{2,2}\!-\!\frac{A_{2,1}}{2}\right)|\nabla V_2^\varepsilon*S_{\mee,2}^{st}\rhottaunne(x)|^2dxdt\\
    &\le\tau\frac{A_{1,2}}{2}\int_{\Rd}|\nabla ( U_{12}^\varepsilon* V_2^\varepsilon*\rhottaune)(x)|^2dx\\
    &\quad+\tau\frac{A_{2,1}}{2}\int_{\Rd}|\nabla (U_{21}^\varepsilon*  V_1^\varepsilon*\rhootaune)(x)|^2dx+ \mee[\rrhotaune]-\mee[\rrhotaunne],
\end{split}
\end{equation*}
thus, by $L^2$ lower semi-continuity of the $H^1$ seminorm,
\begin{equation*}
\begin{split}
    \tau\int_{\Rd}\!\!&\left(A_{1,1}\!-\!\frac{A_{1,2}}{2}\right)|\nabla V_1^\varepsilon*\rhootaunne(x)|^2\!+\!\left(A_{2,2}\!-\!\frac{A_{2,1}}{2}\right)|\nabla V_2^\varepsilon*\rhottaunne(x)|^2dx\\
    &\le\tau\frac{A_{1,2}}{2}\int_{\Rd}|\nabla (U_{12}^\varepsilon*V_2^\varepsilon)*\rhottaune(x)|^2dx\\
    &\quad+\tau\frac{A_{2,1}}{2}\int_{\Rd}|\nabla (U_{21}^\varepsilon* V_1^\varepsilon)*\rhootaune(x)|^2dxdt+ \mee[\rrhotaune]-\mee[\rrhotaunne],
\end{split}
\end{equation*}
By summing up over $n$ from $0$ to $N-1$, taking into account that $x\log x\le x^2$ for any $x\ge0$, Remark \ref{eq:controlbelowentropy} and that second order moments are uniformly bounded (see Proposition \ref{prop:en-ineq-mom-bound-system}), we get
\begin{equation*}
\begin{split}
    \int_0^T\int_{\Rd}&\left(A_{1,1}\!-\!\frac{A_{1,2}}{2}\right)|(\nabla V_1^\varepsilon*\rhootaue(t))(x)|^2+\left(A_{2,2}\!-\!\frac{A_{2,1}}{2}\right)|(\nabla V_2^\varepsilon*\rhottaue(t))(x)|^2\,dx\\
    &\le\mee[\rrho^0]-\mee[\rrhotau^{\varepsilon,N}]+\frac{A_{1,2}}{2}\int_0^{T}\int_{\Rd}|\nabla(V_2^\varepsilon*\rhottaue(t)(x)|^2\,dx\,dt\\
&\quad+\frac{A_{2,1}}{2}\int_0^{T}\int_{\Rd}|\nabla(V_1^\varepsilon*\rhootaue(t))(x)|^2\,dx\,dt\\
&\le\|\rrho^0\|_{L^2}^2+C(\rrho^0,V_1,V_2,A_{1,1},A_{2,2},T)+\frac{A_{1,2}}{2}\int_0^{T}\int_{\Rd}|\nabla(V_2^\varepsilon*\rhottaue(t)(x)|^2\,dx\,dt\\
&\quad+\frac{A_{2,1}}{2}\int_0^{T}\int_{\Rd}|\nabla(V_1^\varepsilon*\rhootaue(t))(x)|^2\,dx\,dt.
\end{split}
\end{equation*}
Weak lower semi-continuity of the norm and \eqref{ass:a-diffusion} give the $H^1$ bound
\begin{equation*}
    \begin{split}
      \left(A_{1,1}-\frac{A_{1,2}}{2}-\frac{A_{2,1}}{2}\right)\int_0^T\int_{\Rd}|\nabla v_1^\varepsilon(t)(x)|^2\,dx\,dt&+\left(A_{2,2}-\frac{A_{1,2}}{2}-\frac{A_{2,1}}{2}\right)\int_0^T\int_{\Rd}|\nabla v_2^\varepsilon(t)(x)|^2\,dx\,dt\\
      &\le\|\rrho^0\|_{L^2}^2+C(\rrho^0,V_1,V_2,A_{1,1},A_{2,2},T).  
    \end{split}
\end{equation*}
\end{proof}

An application of Proposition \ref{prop:aulirs-meas} provides strong convergence in $L^2$, needed to prove convergence to the cross-diffusion system \eqref{eq:cross-diff-sys}, see section \ref{sec:limit-to-cd}.

\begin{prop}\label{prop:strong-convergence-system}
Let $\varepsilon\le1$. There exists a subsequence $\{\bm{v}^{\varepsilon_k}\}_k$ strongly converging to $\bm{v}$ in $L^2([0,T];L^2(\Rd))\times L^2([0,T];L^2(\Rd))$, for any $T>0$.
\end{prop}
\begin{proof}
The proof is similar to the one of Proposition \ref{prop:strong-convergence-v}, see also \cite{DiFranEspFag}, applied to $X:=L^2(\Rd)\times L^2(\Rd)$, $g=d_1$ being the $1$-Wasserstein distance in the product space, and the functional $\mathcal{B}:L^2(\Rd)\times L^2(\Rd)\to[0,+\infty]$ defined
as
\begin{equation*}
\mathcal{B}[\vv]=
\begin{cases}
\sum_{i=1}^2||v_i||_{H^1(\Rd)}^2 + \int_{\R^d}|x|v_i(x) dx, & \text{if } v_i\in \mP_1(\Rd)\cap H^1(\Rd); \\
+\infty & \text{otherwise}.
\end{cases}
\end{equation*}
\end{proof}

\subsection{Towards cross-diffusion systems}\label{sec:limit-to-cd}

Overall the strategy is similar to section \ref{sec:local-limit-pme2}, though we have to clarify how to cope with the cross-interaction terms, leading to cross-diffusion. The right hand side in the Definition of weak measure solution of \eqref{eq:nl-int-sys}, for the first component, can be written as
\begin{align}\label{eq:weak-form-nlis-1st}
    -A_{1,1}\!\int_0^T\!\!\!\!\int_{\Rd} \!\!\nabla \varphi(x)\! \cdot\! \nabla V_1^\varepsilon*v^\varepsilon_{1,t}(x) d\rho^\varepsilon_{1,t}(x)dt\!-\!A_{1,2}\!\int_0^T\!\!\!\!\int_{\Rd}\!\! \nabla \varphi(x)\! \cdot\! (V_1^\varepsilon*U_{12}^\varepsilon)*\nabla v_{2,t}^\varepsilon(x) d\rho_{1,t}^\varepsilon(x) dt.
\end{align}
While for the first integral above we can follow the argument in section~\ref{sec:local-limit-pme2}, applied of course to both components $\rhooe$ and $\rhotwo^\varepsilon$, the cross-interaction part needs a slightly different approach, since $K_1$ is a convolution of three functions. In particular, 
\begin{equation}\label{eq:crossint-excess}
\begin{split}
    \int_0^T\int_{\Rd} \nabla \varphi(x) \cdot (V_1^\varepsilon*U_{12}^\varepsilon)*\nabla v_{2,t}^\varepsilon d\rho_{1,t}^\varepsilon(x) dt&=\int_0^T\int_\Rd (U_{12}^\varepsilon*v_{1,t}^\varepsilon)(x)\nabla\varphi(x)\cdot\nabla v_{2,t}^\varepsilon(x)dxdt\\
    &\quad+\int_0^T\int_\Rd z_{12,t}^\varepsilon(x)\cdot\nabla v_{2,t}^\varepsilon(x)dxdt,
\end{split}
\end{equation}
being, for any $t\in[0,T]$ and $x\in\Rd$,
\begin{equation}
    z_{12,t}^\varepsilon(x):=(U_{12}^\varepsilon*V_1^\varepsilon)* (\rho_{1,t}^\varepsilon\nabla \varphi)(x) -(U_{12}^\varepsilon*v^\varepsilon_{1,t})(x) \nabla \varphi(x).
\end{equation}
The excess term converges to $0$ strongly in $L^2([0,T]\times\Rd)$ by following section \ref{sec:excess-term}, applying the arguments to $P_{12}^\varepsilon:=V_1^\varepsilon*U_{12}^\varepsilon$ instead of $V_\varepsilon$ --- bearing in mind \ref{ass:v1} and \ref{ass:u}. For the readers convenience we remind that the $L^2([0,T];H^1(\Rd))$ bound for $v_i^\varepsilon$ holds true for $U_{ij}^\varepsilon*v_i^\varepsilon$ since, for $i\neq j=1,2$,
\begin{equation}\label{eq:h1-bound-doubleconv}
\begin{split}
&\|U_{ij}^\varepsilon*v_{i,t}^\varepsilon\|_{L^2(\Rd)}\le  \|v_{i,t}^\varepsilon\|_{L^2(\Rd)}=\|v_{i,t}^\varepsilon\|_{L^2(\Rd)};\\
&\|U_{ij}^\varepsilon*\nabla v_{i,t}^\varepsilon\|_{L^2(\Rd)}\le  \|\nabla v_{i,t}^\varepsilon\|_{L^2(\Rd)}=\|\nabla v_{i,t}^\varepsilon\|_{L^2(\Rd)}.
\end{split}
\end{equation}
As in Lemma \ref{lem:limit-dist}, one can prove that the sequence $U_{ij}^\varepsilon*v_i^\varepsilon$ has the same distributional limit of the sequence $v_i^\varepsilon$, i.e. $\rho_i$, for $i\neq j=1,2$.
\begin{lem}\label{lem:limit-dist-cross}
For any $t\in[0,T]$ and any $\varphi\in C_c^1(\Rd)$ it holds, for $i\neq j=1,2$,
$$
\lim_{\varepsilon\to0^+}\int_\Rd \varphi(x)(U_{ij}^\varepsilon*v_{i,t}^\varepsilon)(x)\,dx=\int_{\Rd}\varphi(x)\,d\tilde{\rho}_i(t).
$$
\end{lem}
\begin{proof}
For $t\in[0,T]$ and $\varphi\in C_c^1(\Rd)$, by using that $U_{ij}$ is even and the Definition of $v_{i,t}^\varepsilon$ we obtain:
\begin{align*}
\left|\int_\Rd\varphi(x)(U_{ij}^\varepsilon*v_{i,t}^\varepsilon)(x)\,dx-\int_\Rd\varphi(x)v_{i,t}^\varepsilon(x)dx\right|&=\left|\int_\Rd(\varphi*U_{ij}^\varepsilon)(x)v_{i,t}^\varepsilon(x)dx-\int_\Rd\varphi(x)v_{i,t}^\varepsilon(x)dx\right|\\
&=\left|\int_\Rd[(\varphi*U_{ij}^\varepsilon)(x)-\varphi(x)]v_{i,t}^\varepsilon(x)dx\right|\\
&\le\int_\Rd\int_\Rd|\varphi(x-y)-\varphi(x)|v_{i,t}^\varepsilon(x)d U_{ij}^\varepsilon(y) dx\\
&\le\|\varphi\|_\infty\int_\Rd |y|\, dU_{ij}^\varepsilon(y) \\
&=\varepsilon\|\varphi\|_\infty\int_\Rd |x|\,dU_{ij}(x) ,
\end{align*}
which converges to $0$ as $\varepsilon\to0^+$ since $\int_\Rd|x|\,dU_{ij}(x)<+\infty$.
\end{proof}

Below we state the main result for this section, stressing that in this case uniqueness is missing.

\begin{thm}\label{thm:existence-cross}
Let $\varepsilon\le1$ and $\rrho^0\in(\mpdtard\cap L^2(\Rd))^2$ such that $\mathrm{E}[\rrho^0]<\infty$. The sequence $\{\rrho^\varepsilon\}_\varepsilon$ of solutions to \eqref{eq:nlie} admits a subsequence narrowly converging to a weak solution $\tilde{\rrho}$ of \eqref{eq:cross-diff-sys}. 
\end{thm}
\begin{proof}
The main difference with respect to the proof of theorem \ref{thm:existence-pme2} is with the cross-interaction term as pointed out in \eqref{eq:weak-form-nlis-1st} and \eqref{eq:crossint-excess}. More precisely, in \eqref{eq:crossint-excess} we need to make sure that a subsequence of $U_{12}^\varepsilon*v_{1,t}^\varepsilon$ strongly converges to $\tilde{\rho}_1$ in $L^2([0,T]\times\Rd)$. In view of \eqref{eq:h1-bound-doubleconv}, we can use Proposition \ref{prop:aulirs-meas} applied to a subsequence of $(U_{12}^\varepsilon*v_{1,t}^\varepsilon,U_{21}^\varepsilon*v_{2,t}^\varepsilon)$, exactly as in Proposition \ref{prop:strong-convergence-system}. For each component, the strong $L^2$ limit coincides with $\rho_i$ due to Lemma \ref{lem:limit-dist-cross}.
\end{proof}

\section{Further perspectives}\label{sec:conclusion}

The main contribution of this work is to provide a rigorous analytical derivation of the quadratic porous medium equation and a class of cross-diffusion systems. Our strategy relies on an appropriate time-discretisation of a nonlocal interaction equation (system) in the $2$-Wasserstein space. This is relevant for both the well-posedness of \eqref{eq:pme2} and \eqref{eq:cross-diff-sys}, and their numerical study. We relaxed previous assumptions on the interaction kernel, allowing for pointy potentials, e.g. Morse. As mentioned above, a key motivation for our approach is to provide an analysis that works without geodesic $\lambda$-convexity techniques, having in mind cases where only a JKO-approach may be feasible. Therefore, we could prove a nonlocal-to-local limit for cross-diffusion systems.

\subsection*{Nonconservative Forces}

In this paper, both equations considered have a $2$-Wasserstein gradient flow structure, but our approach may be used even if the PDEs under study are not gradient flows --- this is a considerable advantage of our result. A prototypical example is given by the PDE
\[
\partial_t\rho=\frac{1}{2}\Delta \rho^2+\nabla\cdot(\rho v),
\]  
being $v\neq\nabla\varphi$, for some function $\varphi$. The addition of the \textit{non-gradient flow part} can be overtaken by considering a suitable splitting (JKO) scheme, as in \cite{CL2}. 

The equation above is also significant in the context of networks, where Wasserstein-type metrics have been derived recently (cf. \cite{burger2021dynamic,erbar2021gradient}).

\subsection*{Other exponents}

A natural question may arise is whether our approach can be extended to linear diffusion and $m\neq2$. The first observation to be made is that the approximating equation should be different, for instance the non-viscous version of the one proposed in \cite{FigPhi2008} or \cite[Eq. (8)]{Patacchini_blob19}. While the time-discretisation could be relatively ``easy'' to develop, it may be not trivial to obtain Sobolev bounds in order to obtain compactness. The analysis could be easier if one restricts to a torus, and using a different version of Aubin-Lions Lemma.

\subsection*{Deterministic particle methods}

Last but not least, it is still open to obtain an analytical proof of a deterministic particle approximation for the porous medium equation, as well as linear diffusion. Both this paper and \cite{Patacchini_blob19}, for $m=2$, require initial data to have finite logarithmic entropy, thus excluding particle solutions of the nonlocal interaction equation. Anyway, numerical simulations show that this is not to be excluded, see \cite[Section 6]{Patacchini_blob19}. The main challenge is then to relax the initial assumption on the logarithmic entropy. In this direction, $\lambda$-convexity of the energy plays a key role in \cite{blob_weighted_craig} to rigorously prove a qualitative result when the number of particles, $N$, depends exponentially on $\varepsilon$ and the approximation in Wasserstein of the initial datum. A similar result can be proven also in our case, even for systems, only using $\lambda$-convexity of the nonlocal energy. However, this would narrow the class of cross-diffusion systems obtained as one would need cross-interaction potentials to be proportional. Obtaining a proof for $N$ independent on $\varepsilon$ and quantitative estimates is still open. This is also relevant in the case of cross-diffusion systems as $\lambda$-convexity fails.

\section*{Acknowledgements} 
The authors are grateful to the anonymous reviewers, Nadia Ansini (Sapienza University of Rome), Maria Bruna (University of Cambridge), José Antonio Carrillo (University of Oxford), and Francesco S. Patacchini (IFP Energies Nouvelles) for fruitful discussions. This work was carried out while AE was a postdoctoral researcher at FAU Erlangen-N\"{u}rnberg. The authors thankfully acknowledge support by the German Science Foundation (DFG) through CRC TR 154  ``Mathematical Modelling, Simulation and Optimization Using the Example of Gas Networks''. AE also acknowledges support by the Advanced Grant Nonlocal-CPD (Nonlocal PDEs for Complex Particle Dynamics: Phase Transitions, Patterns and Synchronization) of the European Research Council Executive Agency (ERC) under the European Union’s Horizon 2020 research and innovation programme (grant agreement No. 883363).

\bibliography{references}

\def\cprime{$'$}
\begin{thebibliography}{10}

\bibitem{AGS}
L.~Ambrosio, N.~Gigli, and G.~Savar{\'e}.
\newblock {\em Gradient flows in metric spaces and in the space of probability
  measures}.
\newblock Lectures in Mathematics ETH Z\"urich. Birkh\"auser Verlag, Basel,
  second edition, 2008.

\bibitem{beck_matthes_zizza22}
L.~Beck, D.~Matthes, and M.~Zizza.
\newblock Exponential convergence to equilibrium for coupled systems of
  nonlinear degenerate drift diffusion equations.
\newblock {\em arXiv preprint arXiv:2112.05810}, 2022.

\bibitem{bertozzi}
A.~L. Bertozzi, T.~Laurent, and J.~Rosado.
\newblock {$L^p$} theory for the multidimensional aggregation equation.
\newblock {\em Comm. Pure Appl. Math.}, 64(1):45--83, 2011.

\bibitem{burger2021dynamic}
M~Burger, I.~Humpert, and J.-F. Pietschmann.
\newblock Dynamic optimal transport on networks.
\newblock {\em arXiv preprint arXiv:2101.03415}, 2021.

\bibitem{CL2}
G.~Carlier and M.~Laborde.
\newblock A splitting method for nonlinear diffusions with nonlocal,
  nonpotential drifts.
\newblock {\em Nonlinear Anal.}, 150:1--18, 2017.

\bibitem{CCH14}
J.~A. Carrillo, Y.-P. Choi, and M.~Hauray.
\newblock The derivation of swarming models: mean-field limit and wasserstein
  distances.
\newblock In {\em Collective dynamics from bacteria to crowds}, pages 1--46.
  Springer, 2014.

\bibitem{Patacchini_blob19}
J.~A. Carrillo, K.~Craig, and F.~S. Patacchini.
\newblock A blob method for diffusion.
\newblock {\em Calc. Var. Partial Differential Equations}, 58(2):Paper No. 53,
  53, 2019.

\bibitem{CDFEFS20}
J.~A. Carrillo, M.~Di~Francesco, A.~Esposito, S.~Fagioli, and M.~Schmidtchen.
\newblock Measure solutions to a system of continuity equations driven by
  {N}ewtonian nonlocal interactions.
\newblock {\em Discrete Contin. Dyn. Syst.}, 40(2):1191--1231, 2020.

\bibitem{CDFFLS}
J.~A. Carrillo, M.~Di~Francesco, A.~Figalli, T.~Laurent, and D.~Slepcev.
\newblock Global-in-time weak measure solutions and finite-time aggregation for
  nonlocal interaction equations.
\newblock {\em Duke Math. J.}, 156(2):229--271, 2011.

\bibitem{Holzinger_deriv_cross_diff}
L.~Chen, E.~S. Daus, A.~Holzinger, and A.~J\"{u}ngel.
\newblock Rigorous derivation of population cross-diffusion systems from
  moderately interacting particle systems.
\newblock {\em J. Nonlinear Sci.}, 31(6):Paper No. 94, 38, 2021.

\bibitem{ChenDausJuengel2019}
L.~Chen, E.~S. Daus, and A.~J\"{u}ngel.
\newblock Rigorous mean-field limit and cross-diffusion.
\newblock {\em Z. Angew. Math. Phys.}, 70(4):Paper No. 122, 21, 2019.

\bibitem{Craig_blob2016}
K.~Craig and A.~L. Bertozzi.
\newblock A blob method for the aggregation equation.
\newblock {\em Math. Comp.}, 85(300):1681--1717, 2016.

\bibitem{blob_weighted_craig}
K.~Craig, K.~Elamvazhuthi, M.~Haberland, and O.~Turanova.
\newblock A blob method for inhomogeneous diffusion with applications to
  multi-agent control and sampling.
\newblock {\em arXiv preprint arXiv:2202.12927}, 2022.

\bibitem{DalKen84}
B.~E.~J. Dahlberg and C.~E. Kenig.
\newblock Nonnegative solutions of the porous medium equation.
\newblock {\em Comm. Partial Differential Equations}, 9(5):409--437, 1984.

\bibitem{Daneri_Radici_Runa_JDE}
S.~Daneri, E.~Radici, and E.~Runa.
\newblock Deterministic particle approximation of aggregation-diffusion
  equations on unbounded domains.
\newblock {\em J. Differential Equations}, 312:474--517, 2022.

\bibitem{DS}
S.~Daneri and G.~Savaré.
\newblock Eulerian calculus for the displacement convexity in the wasserstein
  distance.
\newblock {\em SIAM J. Math. Anal.}, 40(3):1104--–1122, 2008.

\bibitem{Degond_Mustieles90}
P.~Degond and F.-J. Mustieles.
\newblock A deterministic approximation of diffusion equations using particles.
\newblock {\em SIAM J. Sci. Statist. Comput.}, 11(2):293--310, 1990.

\bibitem{DiFranEspFag}
M.~Di~Francesco, A.~Esposito, and S.~Fagioli.
\newblock Nonlinear degenerate cross-diffusion systems with nonlocal
  interaction.
\newblock {\em Nonlinear Analysis}, 169:94--117, 2018.

\bibitem{DFESPSCH21}
M.~Di~Francesco, A.~Esposito, and M.~Schmidtchen.
\newblock Many-particle limit for a system of interaction equations driven by
  {N}ewtonian potentials.
\newblock {\em Calc. Var. Partial Differential Equations}, 60(2):Paper No. 68,
  44, 2021.

\bibitem{DFF}
M.~Di~Francesco and S.~Fagioli.
\newblock Measure solutions for nonlocal interaction pdes with two species.
\newblock {\em Nonlinearity}, 26:2777--2808, 2013.

\bibitem{DFM}
M.~Di~Francesco and D.~Matthes.
\newblock Curves of steepest descent are entropy solutions for a class of
  degenerate convection-diffusion equations.
\newblock {\em Calc. Var. Partial Differential Equations}, 50(1--2):199--230,
  2014.

\bibitem{dobrushin}
R.~L. Dobru\v{s}in.
\newblock Vlasov equations.
\newblock {\em Funktsional. Anal. i Prilozhen.}, 13(2):48--58, 96, 1979.

\bibitem{erbar2021gradient}
M.~Erbar, D.~Forkert, J.~Maas, and D.~Mugnolo.
\newblock Gradient flow formulation of diffusion equations in the wasserstein
  space over a metric graph.
\newblock {\em arXiv preprint arXiv:2105.05677}, 2021.

\bibitem{FigPhi2008}
A.~Figalli and R.~Philipowski.
\newblock Convergence to the viscous porous medium equation and propagation of
  chaos.
\newblock {\em ALEA Lat. Am. J. Probab. Math. Stat.}, 4:185--203, 2008.

\bibitem{Fontbona_Meleard_nl_SKT}
Joaquin Fontbona and Sylvie M\'{e}l\'{e}ard.
\newblock Non local {L}otka-{V}olterra system with cross-diffusion in an
  heterogeneous medium.
\newblock {\em J. Math. Biol.}, 70(4):829--854, 2015.

\bibitem{golse}
F.~Golse.
\newblock The mean-field limit for the dynamics of large particle systems.
\newblock In {\em Journ\'{e}es ``\'{E}quations aux {D}\'{e}riv\'{e}es
  {P}artielles''}, pages Exp. No. IX, 47. Univ. Nantes, Nantes, 2003.

\bibitem{GosseToscani06}
L.~Gosse and G.~Toscani.
\newblock Lagrangian numerical approximations to one-dimensional
  convolution-diffusion equations.
\newblock {\em SIAM J. Sci. Comput.}, 28(4):1203--1227, 2006.

\bibitem{JKO}
R.~Jordan, D.~Kinderlehrer, and F.~Otto.
\newblock The variational formulation of the fokker--planck equation.
\newblock {\em SIAM journal on mathematical analysis}, 29(1):1--17, 1998.

\bibitem{JKO98}
R.~Jordan, D.~Kinderlehrer, and F.~Otto.
\newblock The variational formulation of the {F}okker-{P}lanck equation.
\newblock {\em SIAM J. Math. Anal.}, 29(1):1--17, 1998.

\bibitem{Lions_MasGallic_2001}
P.-L. Lions and S.~Mas-Gallic.
\newblock Une m\'{e}thode particulaire d\'{e}terministe pour des \'{e}quations
  diffusives non lin\'{e}aires.
\newblock {\em C. R. Acad. Sci. Paris S\'{e}r. I Math.}, 332(4):369--376, 2001.

\bibitem{MMCS}
D.~Matthes, R.J. McCann, and G.~Savar\'{e}.
\newblock A family of fourth order equations of gradient flow type.
\newblock {\em Comm. P.D.E.}, 34(11):1352--1397, 2009.

\bibitem{morale2005interacting}
D.~Morale, V.~Capasso, and K.~Oelschl{\"a}ger.
\newblock An interacting particle system modelling aggregation behavior: from
  individuals to populations.
\newblock {\em Journal of mathematical biology}, 50(1):49--66, 2005.

\bibitem{morrey1955}
Charles~B. Morrey, Jr.
\newblock On the derivation of the equations of hydrodynamics from statistical
  mechanics.
\newblock {\em Comm. Pure Appl. Math.}, 8:279--326, 1955.

\bibitem{Moussa_nl_lo_triangular2020}
A.~Moussa.
\newblock From nonlocal to classical {S}higesada-{K}awasaki-{T}eramoto systems:
  triangular case with bounded coefficients.
\newblock {\em SIAM J. Math. Anal.}, 52(1):42--64, 2020.

\bibitem{oelschlaeger2001}
K.~Oelschl\"{a}ger.
\newblock A sequence of integro-differential equations approximating a viscous
  porous medium equation.
\newblock {\em Z. Anal. Anwendungen}, 20(1):55--91, 2001.

\bibitem{Ol}
Karl Oelschl\"ager.
\newblock Large systems of interacting particles and the porous medium
  equation.
\newblock {\em J. Differential Equations}, 88(2):294--346, 1990.

\bibitem{onsager1944}
L.~Onsager.
\newblock Crystal statistics. {I}. {A} two-dimensional model with an
  order-disorder transition.
\newblock {\em Phys. Rev. (2)}, 65:117--149, 1944.

\bibitem{Philipowski2007}
R.~Philipowski.
\newblock Interacting diffusions approximating the porous medium equation and
  propagation of chaos.
\newblock {\em Stochastic Process. Appl.}, 117(4):526--538, 2007.

\bibitem{RS}
R.~Rossi and G.~Savar\'{e}.
\newblock Tightness, integral equicontinuity and compactness for evolution
  problems in banach spaces.
\newblock {\em Ann. Sc. Norm. Super. Pisa Cl. Sci. (5)}, 2(2):395–--431,
  2003.

\bibitem{Russo90}
G.~Russo.
\newblock Deterministic diffusion of particles.
\newblock {\em Comm. Pure Appl. Math.}, 43(6):697--733, 1990.

\bibitem{Sandier_Serfaty2004}
E.~Sandier and S.~Serfaty.
\newblock Gamma-convergence of gradient flows with applications to
  {G}inzburg-{L}andau.
\newblock {\em Comm. Pure Appl. Math.}, 57(12):1627--1672, 2004.

\bibitem{S}
F.~Santambrogio.
\newblock {\em Optimal Transport for Applied Mathematicians}, volume~87 of {\em
  Progress in Nonlinear Differential Equations and Their Applications}.
\newblock Birkh\"auser Verlag, Basel, 2015.

\bibitem{Serfaty_gammaconv_2011}
S.~Serfaty.
\newblock Gamma-convergence of gradient flows on {H}ilbert and metric spaces
  and applications.
\newblock {\em Discrete Contin. Dyn. Syst.}, 31(4):1427--1451, 2011.

\bibitem{SKT}
N.~Shigesada, K.~Kawasaki, and E.~Teramoto.
\newblock Spatial segregation of interacting species.
\newblock {\em J. Theoret. Biol.}, 79(1):83--99, 1979.

\bibitem{Vazquez92_PM_introduction}
J.~L. V\'{a}zquez.
\newblock An introduction to the mathematical theory of the porous medium
  equation.
\newblock In {\em Shape optimization and free boundaries ({M}ontreal, {PQ},
  1990)}, volume 380 of {\em NATO Adv. Sci. Inst. Ser. C: Math. Phys. Sci.},
  pages 347--389. Kluwer Acad. Publ., Dordrecht, 1992.

\bibitem{Vaz07}
J.~L. V{\'a}zquez.
\newblock {\em The porous medium equation}.
\newblock Oxford Mathematical Monographs. The Clarendon Press, Oxford
  University Press, Oxford, 2007.
\newblock Mathematical theory.

\bibitem{V2}
C.~Villani.
\newblock {\em Optimal transport : old and new}.
\newblock Grundlehren der mathematischen Wissenschaften. Springer, Berlin,
  2009.

\bibitem{matthes_zinsl}
J.~Zinsl and D.~Matthes.
\newblock Transport distances and geodesic convexity for systems of degenerate
  diffusion equations.
\newblock {\em Calc. Var. Partial Differential Equations}, 54(4):3397--3438,
  2015.

\end{thebibliography}
\bibliographystyle{plain}

\end{document}